\documentclass[12pt]{amsart}

\usepackage{amsaddr}
\usepackage{amsmath}
\usepackage{amssymb}
\usepackage{amscd}
\usepackage{latexsym}
\usepackage{amsthm}
\usepackage{mathrsfs}
\usepackage{color}
\usepackage{amsfonts}
\usepackage{enumitem}
\usepackage{array,tabularx}
\usepackage{caption,setspace}
\usepackage{mathtools}
\usepackage{tikz}
\usepackage{tikz-cd}
\usetikzlibrary{arrows}
\usetikzlibrary{knots}

\AtBeginDocument{%
   \def\MR#1{}
}

\usepackage{stmaryrd,graphicx}
\newtheorem{thm}{Theorem}[section]
\newtheorem*{thm*}{Theorem}
\newtheorem{cor}[thm]{Corollary}
\newtheorem*{cor*}{Corollary}
\newtheorem{lem}[thm]{Lemma}
\newtheorem*{lem*}{Lemma}
\newtheorem{prop}[thm]{Proposition}
\newtheorem*{prop*}{Proposition}

\theoremstyle{definition}
\newtheorem{defn}[thm]{Definition}
\newtheorem*{defn*}{Definition}

\theoremstyle{remark}
\newtheorem{rem}[thm]{Remark}
\newtheorem*{rem*}{Remark}

\newtheorem*{problem*}{Problem}

\numberwithin{equation}{section}

\newcommand{\toiso}{\xrightarrow{\sim}}

\newcommand{\BQ}{\mathbb Q}
\newcommand{\BR}{\mathbb R}
\newcommand{\BC}{\mathbb C}
\newcommand{\BF}{\mathbb F}
\newcommand{\BA}{\mathbb A}
\newcommand{\BI}{\mathbb I}
\newcommand{\BH}{\mathbb H}
\newcommand{\BZ}{\mathbb Z}
\newcommand{\BP}{\mathbb P}
\newcommand{\BFK}{\mathbf K}
\newcommand{\BFR}{\mathbf R}
\newcommand{\Bmu}{\pmb \mu}
\newcommand{\Bf}{\mathbf f}
\newcommand{\Bk}{\mathbf k}
\newcommand{\CA}{\mathcal A}

\newcommand{\CE}{\mathcal E}
\newcommand{\CF}{\mathcal F}
\newcommand{\CG}{\mathcal G}
\newcommand{\CO}{\mathcal O}
\newcommand{\CP}{\mathcal P}
\newcommand{\CM}{\mathcal M}
\newcommand{\CN}{\mathcal N}
\newcommand{\CX}{\mathcal X}
\newcommand{\CY}{\mathcal Y}

\newcommand{\wmax}{\mathrm {max}}
\newcommand{\wmin}{\mathrm {min}}

\newcommand{\ol}{\overline}

\newcommand{\nil}{\mathrm{nil}}
\newcommand{\aut}{\mathrm{aut}}
\newcommand{\eend}{\mathrm{end}}
\newcommand{\wdim}{\mathrm{dim}}
\newcommand{\HLV}{\mathrm{HLV}}
\newcommand{\Sch}{\mathrm{Sch}}

\DeclareMathOperator{\Gr}{Gr}
\DeclareMathOperator{\AGr}{\widehat{Gr}}
\DeclareMathOperator{\Fl}{Fl}
\DeclareMathOperator{\AFl}{\widehat{Fl}}

\DeclareMathOperator{\ASpr}{\widehat{Spr}}
\DeclareMathOperator{\ANilp}{\widehat{Nilp}}
\DeclareMathOperator{\GL}{GL}
\DeclareMathOperator{\Hom}{Hom}
\DeclareMathOperator{\Aut}{Aut}
\DeclareMathOperator{\Iso}{Iso}
\DeclareMathOperator{\Hall}{Hall}
\DeclareMathOperator{\Mat}{Mat}
\DeclareMathOperator{\Lie}{Lie}
\DeclareMathOperator{\weight}{weight}

\DeclareMathOperator{\Coh}{Coh}
\DeclareMathOperator{\Par}{Par}
\DeclareMathOperator{\QPar}{QPar}
\DeclareMathOperator{\ParBun}{ParBun}

\DeclareMathOperator{\ParHiggs}{ParHiggs}
\DeclareMathOperator{\Bun}{Bun}

\DeclareMathOperator{\res}{res}

\DeclareMathOperator{\pExp}{Exp}
\DeclareMathOperator{\pLog}{Log}
\DeclareMathOperator{\Tr}{Tr}

\DeclareMathOperator{\Spec}{Spec}

\DeclareMathOperator{\Sym}{Sym}
\DeclareMathOperator{\Ext}{Ext}
\DeclareMathOperator{\inv}{inv}

\DeclareMathOperator{\ord}{ord}

\DeclareMathOperator{\kernel}{Ker}
\DeclareMathOperator{\image}{Im}
\DeclareMathOperator{\cokernel}{Coker}

\DeclareMathOperator{\Id}{Id}

\DeclareMathOperator{\rank}{rank}
\DeclareMathOperator{\type}{type}

\title[Poincar\'e polynomials of character varieties]{Poincar\'e polynomials of character varieties, Macdonald polynomials and affine Springer fibers}
\author{Anton Mellit}
\email{anton.mellit@univie.ac.at}
\address{Faculty of Mathematics, University of Vienna, \\
Oskar-Morgenstern-Platz 1, 1090 Vienna, Austria}

\begin{document}
\onehalfspacing

\begin{abstract}
We prove an explicit formula for the Poincar\'e polynomials of parabolic character varieties of Riemann surfaces with semisimple local monodromies, which was conjectured by Hausel, Letellier and Rodriguez-Villegas. Using an approach of Mozgovoy and Schiffmann the problem is reduced to counting pairs of a parabolic vector bundles and a nilpotent endomorphism of prescribed generic type. The generating function counting these pairs is shown to be a product of Macdonald polynomials and the function counting pairs without parabolic structure. The modified Macdonald polynomial $\tilde H_\lambda[X;q,t]$ is interpreted as a weighted count of points of the affine Springer fiber over the constant nilpotent matrix of type $\lambda$.
\end{abstract}

\maketitle

\tableofcontents

\section{Introduction} 

\subsection{The conjectures}
Let $\Sigma$ be a Riemann surface of genus $g$ and let $S=(s_1,\ldots,s_k)$ be a collection of $k$ marked points on $\Sigma$. Fix an integer $n>0$ and $k$ conjugacy classes $C_1,\ldots,C_k$ in $\GL_n$. In this paper we study the case when each $C_i$ can be represented by a diagonal matrix with given multiplicities of eigenvalues $r_{i,1}$, $r_{i,2}$, \ldots. Then one defines the character variety $\CM$ as the moduli space of local systems of rank $n$ on $\Sigma\setminus S$ with local monodromy around $s_i$ lying in $C_i$. We will assume that the eigenvalues of $C_i$ are \emph{generic} (see Definition \ref{defn:generic char var}). 

By the non-abelian Hodge theory of Hitchin \cite{hitchin1987self} and Simpson \cite{simpson1992higgs}, \cite{simpson1990harmonic} the character variety as the so-called \emph{Betti moduli space} is identified with the moduli space of Higgs bundles, which is called the \emph{Dolbeault moduli space}. A study of cohomology of this space was initiated by Hausel and Thaddeus for the case of rank $2$ in \cite{hausel2004generators} (without marked points).

The case of rank $3$ with marked points was studied by Garc\'ia-Prada, Gothen and Mu\~noz in \cite{garcia2007betti}. More partial results were obtained by Garc\'ia-Prada, Heinloth and Schmitt in \cite{garcia2011motives}.
 
In \cite{hausel2008mixed} Hausel and Rodriguez-Villegas computed the E-polynomials of character varieties without marked points. Then together with Letellier in \cite{hausel2011arithmetic} they extended their result to the situation with marked points. Contrary to the previous results, their formula was completely general and allowed them to guess a conjectural formula for the full mixed Hodge polynomials of character varieties in terms of modified Macdonald polynomials. In particular, by specialization their formula predicted the Poincar\'e polynomials of the character varieties.

Mozgovoy extended these conjectures in \cite{mozgovoy2012solutions} and gave a conjectural formula for the motives of the Dolbeault moduli spaces.

\subsection{Counting over finite fields}
The method of computing Poincar\'e polynomials of moduli spaces by counting points over finite fields was introduced by Harder and Narasmihan in \cite{harder1974cohomology}. Hausel suggested that the Poincar\'e polynomials of our moduli spaces in the general case can be obtained by counting Higgs bundles over finite fields.

In \cite{schiffmann2014indecomposable} Schiffmann followed this idea and computed the number of stable Higgs bundles on a curve over a finite field (for $q$ sufficiently large, which was later rectified in the follow-up paper with Mozgovoy \cite{mozgovoy2014counting}, \cite{mozgovoy2017counting}) and thus, in particular, obtained a recipe to compute the Poincar\'e polynomials in the case with no marked points.
The formula of Schiffmann was shown to imply the conjectural formula of Hausel and Rodriguez-Villegas by the author in \cite{mellit2017poincare}.

The case with marked points appeared to be more difficult. In particular, there was no known way of obtaining Macdonald polynomials by counting bundles of some sort. In this paper we prove the conjecture of Hausel, Letellier and Rodriguez-Villegas, and in the process we find such an interpretation of Macdonald polynomials. 

In simple terms, the ``baby version'' of our approach is the result (see \cite{macdonald1995symmetric}) that the number of partial flags over a finite field preserved by a fixed nilpotent matrix is given by the Hall-Littlewood polynomial (see Corollary \ref{cor:hall counting flags} for a precise statement adopted to our notations). We think of a vector space with a nilpotent operator as a special case of a vector bundle with a nilpotent endomorphism. Counting flags is analogous to counting parabolic structures at marked points. 

In Theorem \ref{thm: with macdonald} we show how the answer to the problem of  counting triples vector bundle + nilpotent endomorphism + parabolic structure is formulated in terms of Macdonald polynomials. As a bi-product, in Corollary \ref{cor:count gets macdonald} and Theorem \ref{thm:springer fiber} we obtain a purely local interpretation of Macdonald polynomials by counting nilpotent matrices over the power series ring, which is also interpreted as counting points in the affine Springer fiber over a constant nilpotent matrix.

\subsection{Plan of the paper}
An expert reader who knows Mozgovoy-Schiffmann's paper \cite{mozgovoy2014counting} well and believes that it extends directly to the parabolic case is advised to skip the rest of the paper and concentrate on Sections \ref{sec:linear algebra}, \ref{sec:modifications} and \ref{sec:counting vector bundles} where the results on counting bundles with nilpotent endomorphism and the affine Springer fiber (Section \ref{ssec:affine springer}) are shown.

In Section \ref{sec:hall and symmetric} we summarize definitions related to symmetric functions and Hall algebras, mostly to fix notations, but also to remind the reader on how to count flags fixed by a nilpotent matrix and obtain Hall-Littlewood polynomials.

In Section \ref{sec:category parabolic} we study abelian categories of coherent parabolic sheaves and coherent Higgs parabolic sheaves. We compute the Euler pairing, prove Serre duality, and explain Harder-Narasimhan theory. Some results of this section can also be found in the literature in slightly different formulations.

In Section \ref{sec:from to} we proceed in the way parallel to Mozgovoy-Schiffmann \cite{mozgovoy2014counting} to obtain formulas for the numbers of indecomposable parabolic bundles (Corollary \ref{cor:indecomposable}) and stable Higgs bundles (Corollary \ref{cor:counting generic}). We do not need the number of indecomposable parabolic bundles, but we include this result for completeness. Finally, we obtain the Poincar\'e polynomials of our moduli spaces in Theorem \ref{thm:formula poincare} and Theorem \ref{thm:formula poincare charvar}.

\subsection{Main ideas}
We say a few words about intuition behind this paper. The formula of Hausel, Letellier and Rodriguez-Villegas expresses the generating function of mixed Hodge polynomials as follows:
\[
(q-1)(1-t) \pLog\left[\sum_{\lambda\in\CP} \Omega^\HLV_\lambda(q,t,\sigma_\bullet) \prod_{i=1}^k \tilde H_\lambda[X_i;q,t] T^{|\lambda|}\right].
\]
The summation goes over the set of partitions $\lambda$. For each partition we have $\Omega^\HLV_\lambda$ a certain explicit rational function in $q$, $t$ and the Frobenius eigenvalues of the curve $\Sigma$. Then we have the modified Macdonald polynomials $H_\lambda[X_i;q,t]$, each evaluated in it own group of variables. Schiffmann's formula (without marked points, for a different invariant instead of the mixed Hodge polynomial and after a change of variable) looks similar:
\[
(q-1)(1-t) \pLog\left[\sum_{\lambda\in\CP} \Omega^{\Sch}_\lambda(q,t,\sigma_\bullet) T^{|\lambda|}\right].
\]
Here the function $\Omega^{\Sch}_\lambda(q,t,\sigma_\bullet)$ is a deformation of the function $\Omega^\HLV_\lambda(q,t,\sigma_\bullet)$ (see \cite{mellit2017poincare}). To prove the conjecture we try to make sense of the individual term for each partition. In Schiffmann's work each term $\Omega^{\Sch}_\lambda(q,t,\sigma_\bullet)$ corresponds to the weighted number of vector bundles with nilpotent endomorphism where the generic type of the endomorphism (i.e., the Jordan form of its matrix over a generic point of $\Sigma$) is specified by $\lambda$. Now we make a guess that the corresponding term for the parabolic case is given by
\begin{equation}\label{eq:factorization}
\Omega^{\Sch}_\lambda(q,t,\sigma_\bullet)\prod_{i=1}^k \tilde H_\lambda[X_i;q,t]
\end{equation}
At first, this guess seems surprising: using a physical analogy, the marked points do not interact with one another! Using SAGE \cite{sagemath} we counted some numbers of parabolic bundles with endomorphism for $\BP^1$ over the field of two elements, and the computations showed that the guess might be right.

Then the idea of the proof goes as follows: Suppose we know that something like \eqref{eq:factorization} holds with some unknown functions instead of Macdonald polynomials. Then we can try to use some of the properties that uniquely determine Macdonald polynomials to show that the unknown functions are indeed the Macdonald polynomials. Here we use the fact that Macdonald polynomials are uniquely identified by orthogonality with respect to the modified Hall scalar product, together with upper-triangularity with respect to the modified monomial basis. For instance, orthogonality is shown by a direct computation of the total number of triples bundle + parabolic structure + nilpotent endomorphism in the case of $\BP^1$ with two marked points.

So the problem is reduced to showing that the factorization of the form \eqref{eq:factorization} holds. For this we look at the marked points where the nilpotent endomorphism degenerates. We show how the endomorphism together with the underlying bundle can be ``straightened'' in finitely many ways to obtain a new bundle with endomorphism, this time non-degenerate at a given point. Then the counting problem of all endomorphisms can be reduced to counting non-degenerate endomorphisms. For each non-degenerate endomorphism we need to count in how many ways it can be ``de-straightened'' to obtain all degenerate endomorphisms. As it turns out, this last count does not depend on the curve or locations of marked points and we arrive to the formula of the right shape \eqref{eq:factorization}.

\subsection{Search for combinatorial interpretation}
The fact the Hall-Littlewood polynomials count partial flags respected by a nilpotent matrix can be shown purely combinatorially. Recall that the Schur polynomial $s_\lambda$ can be written as a sum $\sum_T x^T$ where $x=(x_1,\ldots,x_N)$ denotes the variables and the sum goes over the set of Young tableaux, which are labelings of the set of cells of $\lambda$ by numbers between $1$ and $N$ satisfying some inequalities for neighboring cells in the horizontal and the vertical direction. Then there is a $q$-deformation of this formula that produces Hall-Littlewood polynomials. Now we only enforce the inequalities in the vertical direction, but count every tableaux with weight $q^{\inv(T)}$, where $\inv(T)$ is the number of inversions. For details the reader is advised to read \cite{haglund2005combinatorial}, where a formula for Macdonald polynomials is given in a similar way. In that formula we also sum over tableaux, but now we do not enforce any inequalities neither in horizontal, nor in vertical direction. Each tableaux comes with a weight, which is a monomial in two variables $q$ and $t$. Setting $t=0$ one obtains the Hall-Littlewood polynomials. In the case of Hall-Littlewood polynomials it can be shown that each tableaux corresponds to a cell in the Schubert decomposition of the partial flag variety, and $q^{\inv(T)}$ counts the number of flags in the cell that are respected by the nilpotent matrix. Thus a natural guess would be that for our counts in the affine Springer fiber one can try to match cells of the corresponding affine Schubert decomposition with terms in the formula for the Macdonald polynomials in \cite{haglund2005combinatorial}. This is an interesting topic for further investigation.

\subsection{A speculation}
Most results of Section \ref{sec:linear algebra} can be extended to other local rings, such as the ring of p-adic integers $\BZ_q$. Thus it would be interesting to see if we can obtain a parallel theory if we replace the curve over a finite field by a number field. What would be the Macdonald polynomials for the primes at infinity, i.e. for $\Bk=\BR$ and $\Bk=\BC$?

\section{Hall algebras and symmetric functions}\label{sec:hall and symmetric}
\subsection{Symmetric functions and lambda rings}\label{subs:symmetric functions}
There are two ways to see symmetric functions (in infinitely many variables). One as functions, as in \cite{macdonald1995symmetric}, another as operations in lambda rings. Choose a base ring $R$. A \emph{symmetric function} $F$ of degree $d$ in $z_1, z_2,\ldots$ is a collection of polynomials $F_N\in R[z_1,\ldots,z_N]$  ($N\in\BZ_{\geq0}$) of degree $d$ invariant under the $S_N$-action and satisfying 
\[
F_{N-1}(z_1,\ldots,z_{N-1}) = F_N(z_1,\ldots,z_{N-1},0).
\]
The symmetric functions form a commutative graded ring, which we denote $\Sym[Z]$, where $Z$ stands for the sequence of the variables $Z=(z_1,z_2,\ldots)$. We can consider symmetric functions in several sequences of variables $Z_1$, \ldots, $Z_k$, which we denote $\Sym[Z_1,\ldots,Z_k]$. We naturally have
\[
\Sym[Z_1,\ldots,Z_k] = \bigotimes_{i=1}^k \Sym[Z_i].
\]
Let us recall the notations
\[
e_n(z_1,z_2,\ldots) = \sum_{i_1<\cdots<i_n} z_{i_1} \cdots z_{i_n},\quad h_n(z_1,z_2,\ldots) = \sum_{i_1\leq\cdots\leq i_n} z_{i_1} \cdots z_{i_n},
\]
\[
p_n(z_1,z_2,\ldots) = \sum_i z_i^n.
\]
For any sequence (in particular, for any partition) $\mu=(\mu_1,\ldots,\mu_m)$ we set
\[
h_\mu = \prod_{i=1}^m h_{\mu_i},
\]
and similarly for $e$ and $p$.
We have
\begin{prop}\label{prop:sym func}
The ring $\Sym[Z]$ is isomorphic to the polynomial rings $R[e_1,e_2,\ldots]$ and $R[h_1,h_2,\ldots]$ via the natural maps $e_i\to e_i(z_1,z_2,\ldots)$, $h_i\to h_i(z_1,z_2,\ldots)$. If $\BQ\subset R$ then $\Sym[Z]$ is isomorphic to $R[p_1, p_2,\ldots]$ in a similar way.
\end{prop}

Next we define lambda rings. The definition is simpler when $\BQ\subset R$, so let us assume that this is the case. A \emph{lambda ring structure} on $R$ is a collection of homomorphisms $p_n:R\to R$ ($n\in\BZ_{>0}$) satisfying
\[
p_1[x]=x,\quad p_n[p_m[x]] = p_{mn}[x]\qquad(m,n\in\BZ_{>0},x\in R).
\]
It is customary to use square brackets for these operations.
A ring together with a lambda ring structure is called a \emph{lambda ring}.
For two lambda rings $R, R'$ a lambda ring homomorphism is a ring homomorphism $\varphi:R\to R'$ such that for all $n\in\BZ_{>0}$ and all $x\in R$ we have $\varphi(p_n[x]) = p_n[\varphi(x)]$. Suppose $R$ is the ring of polynomials or rational functions or power series or Laurent series in some variables $x_1, x_2,\ldots$. Then the \emph{usual lambda ring structure} is defined by
\[
p_n[x_i] = x_i^n.
\]
This, of course, depends on the choice of generators of the ring. The \emph{trivial lambda ring structure} on any ring $R$ can be defined by 
\[
p_n[x] = x\quad(x\in R).
\]

Note that we have used the letters $p_n$ both for the power sum symmetric functions, and for the operations in the lambda ring. This is explained by the following construction. Let $F\in \Sym[Z]$ and $x\in \Lambda$ where $\Lambda$ is some lambda ring containing the base ring $R$. We define the \emph{plethystic action} of $F$ on $x$ as follows: Using Proposition \ref{prop:sym func} write $F=f(p_1,p_2,\ldots)$ and then set
\[
F[x] = f(p_1[x], p_2[x],\ldots).
\]
Then we have $p_n[x]=p_n[x]$, which justifies the abuse of notation. This operation satisfies the following properties:
\begin{equation}\label{eq:plethysm}
(F G)[x] = F[x] G[x], \quad(F+G)[x] = F[x]+G[x],\quad \lambda[x] = \lambda
\end{equation}
\[
(F,G\in\Sym[Z],\lambda\in R, x\in\Lambda).
\]
We equip $\Sym[Z]$ with the usual lambda ring structure. In particular, we have
\[
p_n[p_m(z_1,z_2,\ldots)] = p_{mn}(z_1,z_2,\ldots)\quad (m,n\in\BZ_{>0}).
\]
We finally make another abuse of notation by identifying $Z$ with the sum
\[
Z = p_1(z_1,z_2,\ldots) = \sum_i z_i,
\]
which is justified by the following identity:
\[
p_n[Z] = p_n(z_1,z_2,\ldots).
\]

We have
\begin{prop}\label{prop:free lambda ring}
The lambda ring $\Sym[Z]$ is the \emph{free lambda ring} over $R$ generated by $Z$. This means that for any lambda ring $\Lambda$ containing $R$ and any element $x\in R$ there exists a unique lambda ring homomorphism $\varphi_x:\Sym[Z]\to \Lambda$ which acts identically on $R$ satisfying $\varphi_x(Z)=x$.
\end{prop}
\begin{proof}
For any such $\varphi_x$ we necessarily have $\varphi_x(p_n[Z]) = p_n[x]$. Thus such $\varphi_x$ is unique. To prove existence define $\varphi_x$ by the rule $\varphi_x(F) = F[x]$ for all $F\in\Sym[Z]$. By \eqref{eq:plethysm} this is a ring homomorphism. To show that it is a lambda ring homomorphism we consider any $F\in\Sym[Z]$ and $n\in\BZ_{>0}$. Let $F=f(p_1,p_2,\ldots)$. We have 
\[
p_n[F] = f(p_n,p_{2n},\ldots).
\]
Therefore
\[
\varphi_x(p_n[F]) = f(p_n[x], p_{2n}[x],\ldots) = p_n[f(p_1[x], p_2[x],\ldots)] = p_n[F[x]] = p_n[\varphi_x(F)].
\]
\end{proof}

When our base ring $R$ is itself a lambda ring with a non-trivial lambda ring structure, then we define a lambda ring structure on $\Sym[X]$ in such a way that the inclusion $R\subset \Sym[X]$ is a lambda ring homomorphism. Then Proposition \ref{prop:free lambda ring} extends to this situation. The only difference is that we should require $\Lambda$ to contain $R$ as lambda subring, and in the proof we have
\[
\varphi_x(p_n[F]) = p_n[f](p_n[x], p_{2n}[x],\ldots) = p_n[f(p_1[x], p_2[x],\ldots)] = p_n[F[x]] = p_n[\varphi_x(F)],
\]
where $p_n[f]$ means that we apply $p_n$ to the coefficients of $f$, which are elements of $R$.

\subsection{Plethystic exponential and zeta functions}
Next we demonstrate one application of lambda ring techniques. A plethystic exponential is defined as the following formal series of symmetric functions:
\[
\pExp[Z] = \sum_{n=0}^\infty h_n[Z] = \exp\left(\sum_{n=1}^\infty \frac{p_n[Z]}{n}\right)=\prod_{n=0}^\infty \frac{1}{1-z_i}.
\]
It satisfies the multiplicativity property 
\[
\pExp[Z + Z'] = \pExp[Z]\pExp[Z'],
\]
which is an identity of formal series of elements of $\Sym[Z,Z']$. Using the approach of the previous subsection we can make sense of $\Sym[x]$ for any element $x$ in any lambda ring $\Lambda$ provided that the infinite series makes sense.

Let $X$ be an algebraic variety over $\BF_q$. By a theorem of Dwork \cite{dwork1960rationality}, there exist numbers $a_1,a_2,\ldots,a_m\in\ol\BQ$ and $b_1, b_2,\ldots, b_{m'}\in\ol\BQ$ such that for any $k\in\BZ_{>0}$ we have
\[
|X(\BF_{q^k})| = \sum_{i=1}^m a_i^k - \sum_{i=1}^{m'} b_i^k.
\]
Let $\Lambda=\BQ[x_1,\ldots,x_m,y_1,\ldots,y_{m'}]$ with the usual lambda ring structure. The numbers $a_i, b_i$ define a ring homomorphism $\varphi_X:\Lambda\to \overline\BQ$, which is \emph{not} a lambda ring homomorphism.
Consider the element
\[
\CX = \sum_{i=1}^m x_i - \sum_{i=1}^{m'} y_i\in \Lambda.
\]
Then we have an identity for all $k\in\BZ_{>0}$
\[
|X(\BF_{q^k})| = \varphi_X(p_k[\CX]).
\]
The \emph{zeta function} of $X$ is defined by the following infinite product:
\[
\zeta_X(T) = \prod_{x\in X/\BF_q} \frac{1}{1-T^{\deg x}}\in\BQ[[T]].
\]
Here $x$ goes over the set of \emph{closed} points of $X$ and for each $x$ we denote by $\deg x$ the degree of the residue field of $x$ over $\BF_q$. We have the following alternative ways to write it:
\[
\zeta_X(T) = \exp\left(\sum_{n=1}^\infty \frac{T^n |X(\BF_{q^n})|}n\right) = \varphi_X\left(\exp\left(\sum_{n=1}^\infty \frac{p_n[T \CX]} n\right)\right) = \varphi_X\pExp[T\CX].
\]
We have 
\begin{prop}\label{prop:zeta}
Let $\CY$ be an element in a lambda ring $\Lambda'$. Then we have
\[
\prod_{x\in X/\BF_q} \pExp[p_{\deg x}[\CY T]] = \varphi_X \pExp[T\CX\CY].
\]
In the right hand side we evaluate $\pExp$ in $\Lambda\otimes\Lambda'[[T]]$ and then apply $\varphi_X$ to the components in $\Lambda$.
\end{prop}
\begin{proof}
By Proposition \ref{prop:free lambda ring} it is enough to assume $\Lambda'=\Sym[Z]$ and $\CY=Z=z_1+z_2+\cdots$. Then by multiplicativity of $\pExp$ it is enough to assume $\Lambda'=\BQ[z]$, $\CY=z$. In this case the homomorphism sending $T$ to $z T$ is a lambda ring homomorphism, so we obtain the desired identity from the following identity we have already established:
\[
\prod_{x\in X/\BF_q} \frac{1}{1-T^{\deg x}} = \varphi_X \pExp[T\CX].
\]
\end{proof}

\begin{rem}
We will abuse the notation by writing the identity in Proposition \ref{prop:zeta} as follows:
\[
\prod_{x\in X/\BF_q} \pExp[p_{\deg x}[\CY T]] = \pExp[T\CX\CY].
\]
The reader should remember that to make sense of this identity one needs to first evaluate the write hand side with
\[
\CX = \sum_{i=1}^m x_i - \sum_{i=1}^{m'} y_i\in \Lambda,
\]
treating $x_i$, $y_i$ as formal variables, and only afterwards specialize to the Frobenius eigenvalues.
\end{rem}

\subsection{Reproducing kernels}\label{subs:reproducing kernels}
Another useful property of the plethystic exponential is that it is a reproducing kernel for the Hall scalar product. Suppose we have a scalar product $F,G\to (F,G)$ on $\Sym[X]$ such that functions of different degrees are orthogonal. This is the same thing as a having a scalar product on the degree $d$ part of $\Sym[X]$ for each $d$. Choose any basis of $\Sym[X]$ indexed by partitions $(\alpha_\lambda)_{\lambda\in\CP}$ so that $\deg\alpha_\lambda=|\lambda|$. Let $(\beta_\lambda)_{\lambda\in\CP}$ be the dual basis. The \emph{reproducing kernel} is the infinite sum
\begin{equation}\label{eq:sum basis dual}
\sum_{\lambda\in\CP} \alpha_\lambda[X]\beta_\lambda[Y],
\end{equation}
viewed in the completion of $\Sym[X,Y]$. The basic property of the reproducing kernel is that it does not depend on the choice of basis. One can use this idea in the following way. Suppose sequences of symmetric functions $\alpha_\bullet=(\alpha_\lambda)_{\lambda\in\CP}$ and $\beta_\bullet =(\beta_\lambda)_{\lambda\in\CP}$ satisfy $\deg\alpha_\lambda=\deg\beta_\lambda=|\lambda|$ and suppose the sum \eqref{eq:sum basis dual} is equal to the reproducing kernel. Then we can conclude that $\alpha_\bullet$ and $\beta_\bullet$ are dual bases, i.e. $(\alpha_\lambda, \beta_\mu)=\delta_{\lambda,\mu}$.

Recall that the \emph{Hall scalar product} on $\Sym[X]$ is defined by 
\[
(h_\mu, m_\lambda) = \delta_{\mu\lambda} \qquad(\mu,\lambda\in\CP),
\]
where $m_\lambda$ stands for the monomial symmetric function, which is the sum of all monomials whose ordered list of exponents is $\lambda$. Using infinite product expansion one can show that
\[
\pExp[XY]=\sum_{\lambda\in\CP} h_\lambda[X] m_\lambda[Y],
\]
thus $\pExp[XY]$ is the reproducing kernel for the Hall scalar product. Using another expression for $\pExp$, for instance in terms of the power sum functions it is easy to show that $(p_\lambda,p_\mu)=0$ when $\lambda\neq\mu$ and calculate $(p_\lambda,p_\lambda)$. In this way we can easily obtain the so-called Cauchy formulas for other families symmetric functions and other scalar products.

\subsection{Macdonald polynomials}\label{ssec:macdonald}
The set of all partitions is denoted by $\CP$.
,
and 
\[
h_\mu = \prod_i h_{\mu_i}
\]
for any finite sequence $\mu=(\mu_1,\mu_2,\ldots)$.  Note that $F[1]=(F,h_n)$ for any symmetric function $F$ of degree $n$.
Denote by $M_{\succeq \lambda}$ resp. $M_{\preceq \lambda}$ the subspace of symmetric functions spanned by $m_\mu$ with $\mu\succeq\lambda$ resp. $\mu\preceq\lambda$. Recall that 
\[
\mu\preceq\lambda \qquad\Leftrightarrow\qquad \sum_{i=1}^k\mu_i\leq\sum_{i=1}^k\lambda_i\quad(k=1,2,\ldots).
\]

We will use the \emph{modified Macdonald polynomials}. Our base field is $\BQ(q,t)$, which we endow with the usual lambda ring structure.
\begin{defn}[\cite{garsia1996remarkable}]\label{defn:macdonald}
Macdonald polynomials $\tilde H_\lambda[X;q,t]\in \Sym[X]$ are unique symmetric functions satisfying:
\[
\tilde H_\lambda[(t-1)X]\in M_{\preceq \lambda},\qquad \tilde H_\lambda[(q-1)X]\in M_{\preceq \lambda'},\qquad \tilde H_\lambda[1]=1.
\]
\end{defn}

\begin{rem}
Directly from the definition we deduce the two special cases
\[
\tilde H_{(n)}[X;q,t] = \frac{e_n\left[\frac{X}{q-1}\right]}{e_n\left[\frac{1}{q-1}\right]}
=\frac{h_n\left[\frac{X}{1-q}\right]}{h_n\left[\frac{1}{1-q}\right]},
\]
\[
\tilde H_{1^n}[X;q,t] = \frac{e_n\left[\frac{X}{t-1}\right]}{e_n\left[\frac{1}{t-1}\right]}
=\frac{h_n\left[\frac{X}{1-t}\right]}{h_n\left[\frac{1}{1-t}\right]}.
\]
From the defining property of the Schur functions we deduce specializations
\[
\tilde H_\lambda[X;q,q^{-1}]= \frac{s_\lambda\left[\frac{X}{1-q}\right]}{s_\lambda\left[\frac{1}{1-q}\right]} \qquad (\lambda\in\CP).
\]
\end{rem}

Using reproducing kernels it is easy to see that for any function $S$ and any $F, G\in\Sym[X]$ we have
\[
(F[X], G[SX]) = (F[SX], G[X]).
\]
The $q,t$-scalar product is defined as follows:
\[
(F[X], G[X])_{q,t}:=(F[X], G[(q-1)(1-t) X]).
\]
We have
\begin{prop}
The Macdonald polynomials are orthogonal with respect to the $q,t$-scalar product
\end{prop}
\begin{proof}
Suppose $(\tilde H_\lambda, \tilde H_\mu)_{q,t}\neq 0$. This implies
\[
(\tilde H_\lambda[X], \tilde H_\mu[(q-1)(1-t)X])=  (\tilde H_\lambda[(q-1) X], \tilde H_\mu[(1-t) X]) \neq 0.
\]
We have $\tilde H_\lambda[(q-1) X]\in M_{\preceq\lambda'}$. Hence it is a linear combination of $e_\nu$ with $\nu\succeq\lambda$. Using $e_n[-X]=(-1)^n h_n[X]$ we deduce that $\tilde H_\mu[(1-q) X]$ is a linear combination of $h_\nu$ with $\nu\succeq\lambda$. Hence $\mu\succeq\lambda$. Since the pairing is symmetric, we also have $\mu\preceq\lambda$. Therefore $\lambda=\mu$.
\end{proof} 

Since the Hall pairing is non-degenerate, we have $(\widetilde{H}_\lambda,\widetilde{H}_\lambda)\neq 0$ for any partition $\lambda$. There is an explicit formula
\[
(\tilde H_\lambda, \tilde H_\lambda)_{q,t} = z_\lambda(q,t)= \prod_{a,l}(q^a - t^{l+1})(q^{a+1} - t^l),
\]
where the product runs over the arm- and leg-lengths of the hooks of $\lambda$ (see \cite{macdonald1995symmetric}, \cite{garsia1996remarkable}).

\begin{cor}\label{cor:unique macdonald}
The family of modified Macdonald polynomials is uniquely determined by the orthogonality property $(\tilde H_\lambda, \tilde H_\mu)_{q,t}\neq 0$ for $\lambda\neq\mu$, normalization $\tilde H_\lambda[1]=1$ and any of the two upper-triangularity properties:
\[
\tilde H_\lambda[(t-1)X]\in M_{\preceq \lambda},\qquad \tilde H_\lambda[(q-1)X]\in M_{\preceq \lambda'}.
\]
\end{cor}

By specializing $t=0$ we obtain a version of \emph{Hall-Littlewood polynomials}. 
\[
H_\lambda[X;q] = \tilde H_\lambda[X;q,0].
\]
They can also be defined by setting $t=0$ in Definition \ref{defn:macdonald}, they are orthogonal with respect to 
\[
(F[X],G[X])_q = (F[X], G[(q-1) X]),
\]
and an analogue of Corollary \ref{cor:unique macdonald} holds for them. We have
\[
H_{(n)}[X;q] = \tilde H_{(n)}[X;q,t]= \frac{e_n\left[\frac{X}{q-1}\right]}{e_n\left[\frac{1}{q-1}\right]}
=\frac{h_n\left[\frac{X}{1-q}\right]}{h_n\left[\frac{1}{1-q}\right]},
\]
\[
H_{1^n}[X;q] = h_n[X].
\]

\subsection{Hall algebra}
Let $\CA$ be an abelian category. We assume $\CA$ is embedded as a full abelian subcategory closed under extensions, subobjects and quotient objects into a possibly larger abelian category where all the higher $\Ext$ groups are defined, for instance using injective resolutions. Recall that for any two objects $Y,Z\in\CA$ the abelian group $\Ext^1(Y,Z)$ classifies exact sequences
\begin{equation}\label{eq:short exact sequence}
0\to Z \xrightarrow{f} X \xrightarrow{g} Y \to 0
\end{equation}
up to the natural action of $\Aut(X)$. The stabilizer of a pair $(f,g)$ is identified with the set $\Hom(Y,Z)$. The category is said to be \emph{finitary} if $\Hom(X,Y)$ and $\Ext^1(X,Y)$ are finite sets for all $X,Y\in \CA$. If $\CA$ is finitary the \emph{Hall algebra} $\Hall(\CA)$ is defined as the free $\BQ$-vector space on symbols $[X]$ for $X\in\CA$ modulo relations
\[
[X]=[Y]\qquad\text{if $X\cong Y$ in $\CA$}.
\]
The Hall algebra is endowed with a coproduct
\[
\Delta [X] = \sum_{Z\subset X} [X/Z] \otimes [Z],
\]
which makes sense if $X$ has only finitely many subobjects. If this is not the case, we can still define $\Delta$ on a completion of $\Hall(\CA)$, for instance on the vector space of all functions on the isomorphism classes of objects of $\CA$. There is also a product
\[
[Y]*[Z] = \frac{1}{|\Hom(Y,Z)|} \sum_{\xi\in\Ext(Y,Z)} [X_\xi],
\]
where $X_\xi$ denotes the middle object $X$ in the short exact sequence corresponding to $\xi$. There is a natural non-degenerate pairing
\[
([X],[Y]) = |\Iso(X,Y)|,
\]
where $\Iso(X,Y)\subset\Hom(X,Y)$ is the set of all isomorphisms from $X$ to $Y$. For $X,Y,Z\in\CA$, the number of pairs $(f,g)$ such that $f\in\Hom(Z, X)$, $g\in\Hom(X,Y)$ and the sequence \eqref{eq:short exact sequence} is exact can be calculated in two ways as
\[
(\Delta[X], [Y]\otimes [Z]) = ([X], [Y]*[Z]).
\]
The coproduct is obviously co-associative. To see that the product is associative we can use the fact that the product is dual to the coproduct and the pairing is non-degenerate.
\begin{rem}
There are different conventions on defining Hall algebras, for instance see \cite{ringel1990hall} and \cite{green1995hall}. We chose ours so that it directly reflects questions of counting.
\end{rem}

A result of Green (\cite{green1995hall}, \cite{ringel1996green}) providing a deeper relation between $*$ and $\Delta$ can be stated as follows:
\begin{thm}\label{thm:green}
For any $X,Y\in\CA$ we have
\[
\Delta([X]*[Y]) = \sum \frac{|\Ext^1(X_1,Y_2)|}{|\Hom(X_1,Y_2)|} [X_1] * [Y_1] \otimes [X_2]*[Y_2],
\]
provided $\Ext^2(X_1,Y_2)$ vanish for all $X_1, Y_2$ in the sum above. We have used Sweedler notation $\Delta[X]=\sum [X_1]\otimes [X_2]$ meaning that the sum is a weighted sum over pairs $(X_1, X_2)$, and similarly for $Y$.
\end{thm}

\subsection{One-loop quiver}
Let $\CA$ be the category of nilpotent representations of one-loop quiver over a finite field $\Bk$, i.e. the category of pairs $(V, \theta)$ where $V$ is a finite-dimensional vector space over $\Bk$ and $\theta:V\to V$ is nilpotent. Although the Hall algebra of this category is known, e.g. see \cite{macdonald1995symmetric}, we reproduce its computation here for several reasons. One reason is to make exposition self-contained. Another reason is that this computation is the $t=0$ specialization of our result on counting vector bundles with a nilpotent endomorphism, so the reader may be interested in comparing the two, and possibly find a better unifying approach to them. A third reason is that this is a nice example of a computation which lets the reader to get accustomed to some techniques before we go into more difficult computations. 

The category $\CA$ is finitary and hereditary (i.e. all the higher Exts vanish), with vanishing Euler form
\[
\chi(X,Y) = \dim\Hom(X,Y)-\dim\Ext^1(X,Y) = 0.
\]
In particular, by Theorem \ref{thm:green} we know that $\Hall(\CA)$ is a bi-algebra. By the Jordan form theorem, the complete set of isomorphism classes is given by $\{[N_\lambda]:\lambda\in\CP\}$, where for each partition $\lambda$ we denote by $N_\lambda$ the nilpotent matrix acting on the vector space $\Bk^{|\lambda|}$ with Jordan blocks of sizes $\lambda_1',\lambda_2',\ldots$ for $\lambda'$ the conjugate partition. Using duality and the fact that transposition of matrices does not change the Jordan form, we have that for all $X,Y,Z\in\CA$ the number of short exact sequences
\[
0\to Z\to X\to Y \to 0
\]
equals to the number of short exact sequences
\[
0\to Y\to X\to Z \to 0.
\]
Thus $\Hall(\CA)$ is a commutative and co-commutative bi-algebra.

Denote
\[
E(z)=\sum_{\lambda\in\CP} \frac{[N_\lambda]}{|\Aut(N_\lambda)|} z^{|\lambda|}.
\]
Then we have $(E(z),[X])=z^{\dim X}$ for any object $X\in\CA$. Since the Euler form vanishes, we have for any $X,Y\in\CA$
\[
([X]*[Y], E(z)) = \frac{|\Ext^1(X,Y)|}{|\Hom(X,Y)|} z^{\dim X+\dim Y} = (X, E(z)) ([Y],E(z)).
\]
Hence 
\begin{equation}\label{eq:E multiplicative}
\Delta E(z)=E(z)\otimes E(z).
\end{equation}
Another property of $E(z)$ is 
\begin{prop}\label{prop:count nilpotent}
Let $|\Bk|=q$. Then we have 
\[
(E(u), E(v)) = \pExp\left[\frac{uv}{q-1}\right].
\]
\end{prop}
\begin{proof}
We have
\[
(E(u), E(v)) = 
\sum_{n=0}^\infty (uv)^n \sum_{\lambda\vdash n} \frac{1}{|\Aut(N_\lambda)|} = \CN_q(uv),
\]
where for a finite field of size $q$ we write
\[
\CN_q(T) = \sum_{n=0}^\infty T^n \sum_{\lambda\vdash n} \frac{1}{|\Aut(N_\lambda)|} = \sum_{X\in\CA/\sim} \frac{T^{\dim X}}{|\Aut(X)|}.
\]
Let us show that
\begin{equation}\label{eq:Nq T}
\CN_q(T) = \pExp\left[\frac{T}{q-1}\right].
\end{equation}
The proof is explained in a more combinatorial way in \cite{villegas2007counting}.
Consider the abelian category $\CG$ of pairs $(V, \eta)$ where $V$ is a finite-dimensional vector space and $\eta\in\Aut(V)$. Write
\[
G_q(T) = \sum_{X\in\CG/\sim} \frac{T^{\dim X}}{|\Aut(X)|}.
\]
An isomorphism class of an object of dimension $n$ in $\CG$ is the same thing as a conjugacy class of a matrix in $\GL_n(\BF_q)$. So we obtain
\[
G_q(T) = \sum_{n=0}^\infty T^n \sum_{g\in\GL_n(\BF_q)/\sim} \frac{1}{Z(g)} = \sum_{n=0}^\infty T^n = \frac{1}{1-T}
\]
by the orbit-counting theorem. On the other hand, by Jordan theorem 
\begin{equation}\label{eq:product jordan}
\frac{1}{1-T}=G_q(T) = \prod_{x\in\GL_1/\BF_q} \CN_{q^{\deg x}}(T^{\deg x}),
\end{equation}
where the product is over the set of closed points of the scheme $\GL_1/\BF_q$. Note that there is a unique collection of power series $\CN_{q}(T)$ satisfying the equation above for all $q$ where $q$ runs over the set of prime powers. So it is enough to check that \eqref{eq:Nq T} satisfies \eqref{eq:product jordan}. We have
\[
\prod_{x\in\GL_1/\BF_q} \frac{1}{1-T^{\deg x}} = \zeta_{\GL_1/\BF_q}(T) = \pExp[(q-1)T].
\]
By Proposition \ref{prop:zeta} this implies that for any element $X$ in any lambda ring we have
\[
\prod_{x\in\GL_1/\BF_q} \pExp[p_{\deg x}[XT]] = \pExp[(q-1)XT].
\]
In particular, for $X=\frac{1}{q-1}$ we obtain
\[
\prod_{x\in\GL_1/\BF_q} \pExp\left[\frac{T^{\deg(x)}}{q^{\deg(x)}-1}\right] = \frac{1}{1-T}.
\]
\end{proof}

Choose $N\gg 0$ and define for any $X\in\CA$
\[
I_N([X]) =I_N^{z_1,\ldots, z_N}([X])= \left([X], \prod_{i=1}^{N} E(z_i)\right) = \left(\Delta^{N-1}[X], \bigotimes_{i=1}^{N} E(z_i)\right). 
\]
This is a polynomial of degree $\dim X$, which by commutativity of $\Hall(\CA)$ is symmetric in $z_i$. Setting $z_N=0$ produces $I_N$. Thus there is a symmetric function in infinitely many variables which specializes to $I_N$ for any $N$. We denote this symmetric function by $I([X])$. Thus we obtain a map
\[
I:\Hall(\CA) \to \Sym[Z],
\]
where $\Sym[Z]$ denotes the ring of symmetric functions in the alphabet $(z_1,z_2,\ldots)$. More explicitly, we have
\[
I([V,\theta]) = \sum_{\lambda\vdash \dim V} m_\lambda[Z] |\{V_1\subset V_2\subset\cdots V_{l(\lambda)}=V:\theta V_i=V_i,\dim V_i/V_{i-1}=\lambda_i\}|.
\]
Equivalently, $(I([V,\theta]), h_\mu)$ is the number of partial flags of type $\mu$ fixed by $\theta$ for any sequence $\mu=(\mu_1,\ldots,\mu_k)$.

\begin{thm}\label{thm:hall algebra}
The map $I$ is an isomorphism of bi-algebras with pairing, where the coproduct on the ring of symmetric functions is defined by sending the power sum $p_n$ to $p_n\otimes 1 + 1\otimes p_n$ for all $n>0$ and the pairing is the $q$-modified Hall pairing $(F,G)_q = (F[(q-1)X], G[X])$. We have $I([N_\lambda]) = H_\lambda[Z;q]$ for every partition $\lambda$, where $H_\lambda[Z;q]$ is the $t=0$ specialization of the modified Macdonald polynomial $\tilde H_\lambda[Z;q,t]$.
\end{thm}
\begin{proof}
For any $X,Y\in\CA$ we have
\[
I_N([X]*[Y]) = \left([X]\otimes [Y], \Delta \left(\prod_{i=1}^{N} E(z_i)\right)\right) = I_N([X]) I_N([Y])
\]
by \eqref{eq:E multiplicative}.
Hence we have
\[
I([X]*[Y]) = I([X]) I([Y]).
\]
We identify $\Sym[Z]\otimes\Sym[Z]$ with the ring $\Sym[Z,Z']$ of symmetric functions in two sets of variables $Z$, $Z'$. Then the coproduct on symmetric functions can be written as $\Delta F[Z]=F[Z+Z']$. For any $X\in\CA$ we have
\[
\left(I_N^{z_1,\ldots, z_N}\otimes I_N^{z_1',\ldots,z_N'}\right)(\Delta[X]) = \left([X], \prod_{i=1}^{N} E(z_i) \prod_{i=1}^{N} E(z_i')\right)
\]
\[
 = I_{2N}([X])(z_1,\ldots,z_N,z_1',\ldots,z_N').
\]
Thus 
\[
I(\Delta [X]) = \Delta I([X]).
\]
So we have shown that $I$ preserves product and coproduct. Then we calculate $I(E(u))$ as follows:
\[
I_N(E(u)) = \left(\Delta^{N-1} E(u), \bigotimes_{i=1}^{N} E(z_i)\right) = \prod_{i=1}^N (E(u), E(z_i)) = \pExp\left[\frac{u \sum_{i=1}^N z_i}{q-1}\right],
\]
\[
I(E(u)) = \pExp\left[\frac{u Z}{q-1}\right].
\]
The coefficients of series of the form $I\left(\prod_{i=1}^N E(u_i)\right)$ are the symmetric functions $h_\mu\left[\frac{Z}{q-1}\right]$. In particular, they span $\Sym[Z]$. Therefore $I$ is surjective. Since $I$ preserves the degree and the dimensions of degree $n$ parts of $\Hall(\CA)$ and $\Sym[Z]$ are both given by the number of partitions of size $n$, we deduce that $I$ is an isomorphism. So we know that the coefficients of series of the form $\prod_{i=1}^N E(u_i)$ span $\Hall(\CA)$. In order to verify that $I$ preserves the pairing it is enough to consider the following for all $X\in\CA$:
\[
\left([X], \prod_{i=1}^N E(u_i)\right) = I([X])(u_1,\ldots,u_N),
\]
\[
\left(I([X]), I\left(\prod_{i=1}^N E(u_i)\right)\right)_q = \left(I([X]), \pExp\left[\frac{Z \sum_{i=1}^N u_i}{q-1}\right]\right)_q = I([X])(u_1,\ldots,u_N).
\]
So we have shown that $I$ preserves the pairing.

Finally, let us show $I([N_\lambda]) = H_\lambda[Z;q]$. So far we know that $I([N_\lambda])$ are orthogonal with respect to the $q$-deformed Hall scalar product, which is the $t=0$ specialization of the $q,t$-deformed Hall scalar product. The normalization condition $(I([N_\lambda]), h_n)=1$ is evident. So it is enough to show that $I([N_\lambda])$ satisfy one of the two upper-triangularity properties obtained by specialization $t=0$ from the corresponding properties for the modified Macdonald polynomials (see Corollary \ref{cor:unique macdonald}). The matrix $N_{1^n}$ is a single Jordan block. For each $\mu$ there is exactly one flag of type $\mu$ fixed by $N_{1^n}$. Therefore $I(N_{1^n})=h_n[Z]$. For any partition $\lambda$ we have
\[
h_\lambda[Z] = I\left(\prod_{i=1}^{l(\lambda)} [N_{1^{\lambda_i}}] \right).
\]
Note that for each object $(V,\theta)$ appearing in $\prod_{i=1}^{l(\lambda)} [N_{1^{\lambda_i}}]$ we have
\[
\dim\kernel \theta^k \leq \sum_{i=1}^{l(\lambda)} \dim\kernel N_{1^{\lambda_i}}^k = \sum_{i=1}^{l(\lambda)} \min(\lambda_i, k).
\]
This implies
\[
\left(\prod_{i=1}^{l(\lambda)} [N_{1^{\lambda_i}}], [N_\mu]\right)\neq 0 \;\Rightarrow\; \mu \preceq \lambda' \Leftrightarrow \lambda\preceq\mu',
\]
therefore
\[
\left(I([N_\mu])[(q-1)Z], h_\lambda[Z]\right) = (I([N_\mu]), h_\lambda[Z])_q\neq 0 \;\Rightarrow\; \lambda\preceq\mu',
\]
so the expansion of $I([N_\mu])[(q-1)Z]$ in the monomial basis contains only monomials $m_\lambda$ with $\lambda\preceq\mu'$, which is precisely one of the upper-triangularity properties of the modified Macdonald polynomials, which does not change under specialization $t=0$.
\end{proof}

\begin{cor}\label{cor:hall counting flags}
The number of partial flags of type $\mu$ fixed by $N_\lambda$ over $\BF_q$ equals to $(H_\lambda[Z;q],h_\mu[Z])$.
\end{cor}

\section{Linear algebra over the power series ring}\label{sec:linear algebra}
\subsection{Notations}
For any commutative ring $R$ and $m,n\in\BZ_{\geq 0}$ we denote by $\Mat_{m\times n}(R)$ resp. $\Mat_n(R)$ the module of $m\times n$ matrices resp. $n\times n$ matrices over $R$. We denote by $\GL_n(R)$ the group of invertible matrices, i.e. matrices $g\in\Mat_n(R)$ such that $\det g$ is invertible in $R$. For a field $\Bk$ we denote by $\Bk[[x]]$ the power series ring and by $\Bk((x))$ the field of Laurent series. When $\Bk$ is clear from the context we denote 
\[
\BFK = \Bk((x)),\qquad \BFR = \Bk[[x]].
\]
The \emph{degree} of a matrix $g\in\GL_n(\BFK)$ is defined as $\deg g = \ord\det g$. For $N\geq 0$ we denote by $\Mat_n^{\geq -N}(\BFK)$ the set of matrices whose entries have poles of order at most $N$. We write
\[
\GL_n^+(\BFK) = \GL_n(\BFK)\cap\Mat_n(\BFR),
\]
\[
\GL_n^{\geq -N}(\BFK) = \GL_n(\BFK)\cap\Mat_n^{\geq -N}(\BFR).
\]
We denote by $\GL_{n,d}(\BFK)$, $\GL_{n,d}^+(\BFK)$,  $\GL_{n,d}^{\geq -N}(\BFK)$ the corresponding subsets of matrices of degree $d$. We define the \emph{affine Grassmanian} as
\[
\widehat\Gr_n(\Bk) = \GL_n(\BFK)/\GL_n(\BFR),
\]
and similarly $\widehat\Gr_n^+(\Bk)$, $\widehat\Gr_n^{\geq -N}(\Bk)$, $\widehat\Gr_{n,d}(\Bk)$, $\widehat\Gr_{n,d}^+(\Bk)$, $\widehat\Gr_{n,d}^{\geq -N}(\Bk)$. Note that all these quotients are right. Sometimes it is convenient to consider left quotients too, which we will indicate by subscript $L$, e.g.
\[
\widehat\Gr_n(\Bk)_L = \GL_n(\BFR) \backslash \GL_n(\BFK).
\]
Of course, the left and the right quotients can be identified by the transposition.

The set of $\GL_n(\BFR)$-conjugacy classes of nilpotent matrices is denoted as
\[
\ANilp_n(\Bk) = \{\theta\in\Mat_{n}(\BFR)\;\text{nilpotent}\}/\{\theta\sim g \theta g^{-1}:\, g\in\GL_n(\BFR)\}.
\]
\subsection{Basic facts}
Recall that $\BFR$ is Noetherian, which means that every submodule of a finitely generated module is finitely generated. Also $\BFR$ is a principal ideal domain, which implies that every finitely generated torsion-free module is free.  When dealing with matrices over $\BFR$ we will use the following:
\begin{prop}[Hermite normal form]\label{prop:hermite normal form}
For any $M\in\Mat_{m\times n}(\BFR)$ there exists $g\in \GL_m(\BFR)$ such that the matrix $M'=g M$ is upper triangular ($M'_{i,j}=0$ for $i>j$). 
\end{prop}

\begin{prop}[Smith normal form]\label{prop:smith normal form}
For any $M\in\Mat_{m\times n}(\BFR)$ there exists $g_1\in\GL_m(\BFR)$ and $g_2\in\GL_n(\BFR)$ such that $M'=g_1 M g_2$ is diagonal ($M'_{i,j}=0$ for $i\neq j$) and $\ord M'_{i+1,i+1}\geq \ord M'_{i,i}$.
\end{prop}

In both normal forms the orders of the diagonal entries $\ord M'_{i,i}$ do not depend on the choices of $g$ resp. $g_1, g_2$.

We have 
\begin{prop}\label{prop:finiteness grass}
If $\Bk$ is finite, then $\widehat\Gr_{n,d}^+(\Bk)$, $\widehat\Gr_{n,d}^{\geq -N}(\Bk)$ are finite sets for all $n,d,N$. Every equivalence class in $\widehat\Gr_n(\Bk)$ can be represented by a matrix whose entries are Laurent polynomials.
\end{prop}
\begin{proof}
This is well-known, but for convenience of a computationally inclined reader we give an explicit proof. We have that multiplication by $x^N$ gives a bijection $\widehat\Gr_{n,d}^{\geq -N}(\Bk)\cong \widehat\Gr_{n,d+Nn}^+(\Bk)$ so it is enough to consider $\widehat\Gr_{n,d}^+(\Bk)$.
We start with a matrix $g\in \Mat_n(\BFR)$ and let $g_i$ denote the $i$-th row:
\[
g = \begin{pmatrix}g_1\\ \vdots\\ g_n\end{pmatrix}.
\]
Applying Hermite normal form, Proposition \ref{prop:hermite normal form}, to the transpose of $g$ we see that any point of $\widehat\Gr_n$ can be represented by a lower-triangular matrix. Moreover, we can assume the diagonal entries are monomials $x^{m_1},\ldots,x^{m_n}$. We have $\sum m_i = d$, $m_i\geq 0$, so the number of possibilities for the tuple $(m_1,\ldots,m_n)$ is finite. Now suppose $g$ is in such form. The subset of matrices $h\in \GL_n(\BFR)$ such that $gh$ is again in this form is precisely the set of lower-triangular matrices of power series with $1$ on the diagonal. It is clear that multiplying by such $h$ is equivalent to adding to each $a_{i,j}$ an arbitrary series from $t^{m_i}\Bk[[t]$. Summarizing, the complete set of invariants of $g$ modulo $\GL_n(\Bk[[t]])$ is the integers $m_1,\ldots,m_n$ and the elements $a_{i,j}\in \Bk[[t]]/t^{m_i} \Bk[[t]]$ for $i>j$. So the set of possibilities is finite. We give an explicit formula of its size. Let $q=|\Bk|$. Then we have
\[
\sum_{d=0}^\infty t^d |\widehat\Gr^+_{n,d}(\Bk)| = \sum_{m_1,\ldots,m_n=0}^\infty  q^{m_2+2m_3+\cdots+(n-1)m_n} t^{\sum m_i} = \frac{1}{(1-t)(1-q t)\cdots(1-q^{m-1}t)}.
\]
\end{proof}

\subsection{Normal form of a nilpotent matrix}
We need to classify nilpotent matrices over $\BFR$ up to conjugation by $\GL_n(\BFR)$.

First note that a matrix over $\BFR$ is in particular a matrix over $\BFK$, and nilpotent matrices can be transformed into Jordan form over any field. The sizes of Jordan blocks form an invariant of a matrix. We prefer to consider a slightly different version of this:
\begin{defn}
Let $\theta$ be a nilpotent $n\times n$ matrix over a field. Its type $\type(\theta)$ is defined as a partition $\type(\theta)=(\lambda_1\geq\lambda_2\geq\cdots)$ where $\lambda_i = \dim \kernel \theta^i - \dim \kernel \theta^{i-1}$. The type of a matrix over $\BFR$ is defined as the corresponding type over $\BFK$, or equivalently by $\lambda_i = \rank \kernel \theta^i - \rank \kernel \theta^{i-1}$ where we have that $\kernel \theta^{i}$ is free for every $i$. The set of $\GL_n(\BFR)$-conjugacy classes of nilpotent matrices of type $\lambda$ is denoted by $\ANilp_\lambda(\Bk)\subset\ANilp_n(\Bk)$.
\end{defn}

For instance, if $\theta=0$ we have $\type(\theta)=(n)$. In general, if we transform $\theta$ to the Jordan form over a field then the sizes of the blocks are given by the conjugate partition $\type(\theta)'$. 
 
To every nilpotent $\theta\in\Mat_n(\BFR)$ we associate the \emph{kernel filtration}:
\[
\BFR^n = \kernel \theta^0 \supset \kernel \theta^1 \supset \kernel \theta^1\supset \cdots.
\]
We have an exact sequence
\begin{equation}\label{eq:kernel exact sequence}
0\to \kernel \theta^k \to \kernel\theta^{k+1} \to \image \theta^k|_{\kernel\theta^{k+1}} \to 0
\end{equation}
for every $k$. Therefore $\kernel\theta^{k+1}/\kernel\theta^{k}$ is free, and so the exact sequence splits for every $k$. Thus we obtain that any nilpotent matrix can be conjugated to a block-upper-triangular form, which we call a \emph{kernel form}:
\begin{equation}\label{eq: block upper triangular}
\theta=\begin{pmatrix}
0 & \theta_{1,2} & \theta_{1,3} & \cdots & \theta_{1,m}\\
0 & 0 & \theta_{2,3}& \cdots & \theta_{2,m}\\
\vdots & \vdots& \vdots & \cdots & \vdots\\
0 & 0 & 0 & \cdots & \theta_{m-1,m}\\
0 & 0 & 0 & \cdots & 0
\end{pmatrix},
\end{equation}
with blocks of sizes $\lambda_1,\lambda_2,\ldots,\lambda_m$ where $\lambda=\type(\theta)$ and $\theta_{i,i+1}$ has maximal rank $\lambda_{i+1}$ for all $i$.

Then we have a mild strengthening of this:
\begin{prop}\label{prop:strong kernel form}
For any nilpotent $\theta\in\Mat_n(\BFR)$ there exists $g\in\GL_n(\BFR)$ such that $\theta^g = g \theta g^{-1}$ is in the kernel form as above with each $\theta_{i,i+1}$ upper-triangular.
\end{prop}
\begin{proof}
We may assume $\theta$ is already in a kernel form. Then we will apply a block-diagonal matrix $g$ with blocks $g_1, g_2, \ldots, g_m$. So $\theta_{i,i+1}$ is transformed to $g_i \theta_{i,i+1} g_{i+1}^{-1}$. We see that in order to achieve the result we can start with $g_m=\Id$, and then choose $g_{m-1}, g_{m-2}, \ldots$ one by one, at each step applying Proposition \ref{prop:hermite normal form} to transform $\theta_{i,i+1}$ into a Hermite normal form.
\end{proof}

Next we allow transformations by matrices in $\GL_n(\BFK)$ which are not completely arbitrary:
\begin{defn}
Let $\theta\in\Mat_{n}(\BFR)$ be a nilpotent matrix. A matrix $g\in\GL_n(\BFK)$ is \emph{kernel-strict} if restricts to an isomorphism
\[
\kernel \theta \toiso (\kernel g \theta g^{-1})\cap \BFR^n.
\]
\end{defn}
For instance, if both $\theta$ and $g\theta g^{-1}$ are in the kernel form, then kernel-strictness means that $g$ has block form
\[
\begin{pmatrix}
g_{1,1} & g_{1,2}\\
0 & g_{2,2}
\end{pmatrix}
\]
with $g_{1,1}\in\GL_{\lambda_1}(\BFR)$.

For a partition $\lambda$ we denote by $N_\lambda$ the matrix composed of blocks $N_{i,j}$ of sizes $\lambda_i\times \lambda_j$ with all blocks zero except 
\[
N_{i,i+1} = \begin{pmatrix}\Id_{\lambda_{i+1}} \\ 0_{(\lambda_{i}-\lambda_{i+1})\times \lambda_{i+1}}\end{pmatrix}.
\]
We call it the \emph{standard nilpotent matrix of type $\lambda$}.
We then have the following result:
\begin{lem}\label{lem:classify}
Let $\theta\in\Mat_n(\BFR)$ be nilpotent of type $\type(\theta)=\lambda$. Then there exists a kernel-strict $g\in\GL_n(\BFK)$ such that $g\theta g^{-1}=N_\lambda$. We have that $d=\deg g$ does not depend on the choice of $g$, and $d\geq 0$. There exists a bound $N(\lambda, d)$ depending only on $\lambda$ and $d$ such that $g$ can be chosen in $\GL_{n,d}^{\geq -N(\lambda, d)}$.
\end{lem}
\begin{proof}
We may assume that $\theta$ is already in the form of Proposition \ref{prop:strong kernel form}. Then we conjugate $\theta$ by a block-diagonal matrix $h$ with blocks $h_1, h_2, \ldots, h_m$ so that $\theta_{i,i+1}$ becomes $h_i \theta_{i,i+1} h_{i+1}^{-1}$. We will have each $h_i$ upper-triangular matrix with entries in $\BFR$. First we set $h_1=\Id_{\lambda_1}$. Then let $h_2$ be the top $\lambda_2\times\lambda_2$ block of $\theta_{1,2}$. Proceeding in this way we will define $h_i$ to be the top $\lambda_i\times\lambda_i$ block of $h_{i-1} \theta_{i-1,i}$, which is an upper-triangular $\lambda_{i-1}\times\lambda_i$ matrix so that its lower $(\lambda_{i}-\lambda_{i-1})\times \lambda_{i}$ block is guaranteed to be zero. As a result we obtain a block-diagonal matrix $h$ such that $\theta^h=h \theta h^{-1}$ is in the kernel form with blocks $\theta^h_{i,i+1}=N_{i,i+1}$:
\[
\theta^h=\begin{pmatrix}
0 & N_{1,2} & \theta_{1,3}^h & \cdots & \theta_{1,m}^h\\
0 & 0 & N_{2,3}& \cdots & \theta_{2,m}^h\\
\vdots & \vdots& \vdots & \cdots & \vdots\\
0 & 0 & 0 & \cdots & N_{m-1,m}\\
0 & 0 & 0 & \cdots & 0
\end{pmatrix}.
\]
Let $d = \ord\det h$. We have $h\in\Mat_n(\BFR)$ and $h^{-1}\in x^{-d} \Mat_n(\BFR)$. Therefore $\theta^h\in x^{-d} \Mat_n(\BFR)$. Next we conjugate $\theta h$ by block upper triangular matrices $f$ satisfying $f_{i,i}=\Id_{\lambda_i}$. Pick $i,j$ with $i<j-1$ and a $\lambda_i\times \lambda_j$ matrix $C$. Let $f$ be such that $f_{i,j}=C$, all the diagonal blocks are identity, and the rest are zero. Then conjugation by $f$ does the following: It takes a block row $j$, multiplies it on the left by $C$ and adds it to the block row $i$. Then it takes the block column $i$, multiplies it on the right by $C$ and subtracts it from the block row $j$. Thus we can turn to zero the blocks $\theta^h_{i,m}$ with $i<m-1$ modifying only blocks $\theta^h_{i,j}$ with $j<m$. Proceeding in this fashion for each $j=m,m-1,\ldots$ we will turn $\theta^h_{i,j}$ to zero modifying only blocks $\theta^h_{i',j'}$ with $j'<j$, so that the blocks with $j'>j$ remain zero.

In the end we set $g$ to be the product of $h$ and all these matrices $f$ so that we obtain $g \theta g^{-1} = N_\lambda$. The matrix $g$ is block-upper triangular with $g_{1,1}=h_{1,1}=\Id_{\lambda_1}$, hence it is kernel-strict. We have $\deg g = \deg h = d\geq 0$. The poles of $\theta^h$ have orders bounded by $d$. Thus we can bound the maximal order of the poles of the matrix $C$ in the above procedure at each step independently of $\theta$.

Finally, let us show that $d$ does not depend on the choice of $g$. Suppose $g, g'\in\GL_n(\BFK)$ are kernel-strict matrices such that $g \theta g^{-1} = g' \theta g'^{-1} = N_\lambda$. Then $h=g' g^{-1}$ is kernel-strict and commutes with $N_\lambda$. Hence $h$ must preserve the kernel filtration, which means that $h$ is block-upper-triangular. Denote the diagonal blocks of $h$ by $h_i$. We have that $h_i N_{i,i+1} = N_{i,i+1} h_{i+1}$. In particular, the set of entries of $h_{i+1}$ is a subset of the set of entries of $h_i$ for each $i$. We have $h_1\in\GL_{\lambda_1}(\BFR)$ because $h$ is kernel-strict. Thus all entries of $h_i$ are in $\BFR$. Applying the same argument for $h^{-1}$ we see that all entries of $h_i^{-1}$ are also in $\BFR$. Therefore $h_i\in\GL_{\lambda_i}(\BFR)$, so its determinant has order $0$. The determinant of $h$ is the product of the determinants of $h_i$, so we have $\deg h = 0$. Therefore $\deg g = \deg g'$.
\end{proof}

\begin{defn}
We define the \emph{degree} of nilpotent $\theta\in\Mat_n(\BFR)$ as $\deg g$ in the above construction. Thus we obtain a function $\deg: \ANilp_n(\Bk) \to \BZ_{\geq 0}$.
\end{defn}

\begin{cor}[Of the proof]\label{cor:non-degenerate}
Let $\theta\in \Mat_n(\BFR)$ be nilpotent of type $\lambda$ and $\theta(0)\in\Mat_n(\Bk)$ its specialization. The following conditions are equivalent:
\begin{enumerate}
\item $\theta \sim N_\lambda$ in $\ANilp_n(\Bk)$.
\item $\type(\theta(0))=\lambda$,
\item $\deg \theta=0$,
\end{enumerate}
\end{cor}
\begin{proof}
Clearly (i) implies (ii). To show that (ii) implies (iii) we note that in general $(\kernel\theta^i) \otimes_\BFR \Bk\subset \kernel \theta^i(0)$, so $\type(\theta(0))=\type\theta$ implies equality of dimensions, and hence equality $(\kernel\theta^i) \otimes_\BFR \Bk\subset \kernel \theta^i(0)$. Thus if $\theta$ is in a kernel form, then $\theta(0)$ is too. In particular, the diagonal entries of $\theta_{i,i+1}(0)$ in the above proof are non-zero. Thus the diagonal entries of $h(0)$ in the proof of Lemma \ref{lem:classify} are non-zero. Therefore $d=\ord\det h=0$.

To show that (iii) implies (i) we note that $h\in\Mat_n(\BFR)$ in the proof of Lemma \ref{lem:classify}, so $\ord\det h=d=0$ implies $h\in\GL_n(\BFR)$. Hence we have $\theta^h\in\Mat_n(\BFR)$. This shows that the matrices $f$ involved in the construction are also in $\GL_n(\BFR)$. Therefore $g\in\GL_n(\BFR))$.
\end{proof}

\begin{defn}
A nilpotent matrix $\theta$ satisfying the conditions of the above Corollary is said to be \emph{non-degenerate}.
\end{defn}

Unfortunately, we do not have an \emph{explicit} complete classification of nilpotent matrices. However, we have

\begin{cor}
Suppose $\Bk$ is a finite field. Then the set of $\GL_n(\BFR)$-conjugacy classes of nilpotent matrices over $\BFR$ of type $\lambda$ and degree $d$ is finite for all $\lambda$ and $d$.
\end{cor}
\begin{proof}
By Lemma \ref{lem:classify}, any nilpotent matrix $\theta$ of type $\lambda$ and degree $d$ can be obtained as $g^{-1} N_\lambda g$ with $g\in \GL_{n,d}^{\geq -N(\lambda, d)}(\BFK)$. Multiplying $g$ on the right by elements of $\GL_n(\BFR)$ produces equivalent matrices, so the number of nilpotent matrices of given type and degree does not exceed the number of elements of $\widehat\Gr_{n,d}^{\geq -D(\lambda, d)}(\Bk)$, which is finite by Proposition \ref{prop:finiteness grass}.
\end{proof}

\subsection{Classification data}\label{ssec:classification data}
Let $\Bk$ be a finite field. As we mentioned earlier, we do not have an explicit classification of nilpotent matrices over $\BFR$, but we can choose a classification in the sense explained below.
For each element of $\ANilp_n(\Bk)$ we first pick a representative $\theta$. Let $\lambda=\type \theta$, $d=\deg \theta$. Let $Z(\theta)$ denote the centralizer of $\theta$ inside $\GL_n(\BFR)$. Then we pick an orbit of the group  $Z(N_{\lambda})\times Z(\theta)$ naturally acting on the set of $g\in\GL_n(\BFK)$ such that $g \theta g^{-1} = N_{\lambda}$, $g$ is kernel-strict and the order of poles of $g$ is as small as possible, as in Lemma \ref{lem:classify}. The choice of such an orbit for each conjugacy class will be called a \emph{classification data} for $\Bk$. Picking a representative $g_\theta$ of such an orbit allows us to write the orbit as $M_\theta:=Z(N_{\lambda}) g_\theta Z(\theta)$. When $\theta'$ is another representative of the same conjugacy class, so that $\theta'=h \theta h^{-1}$ for $h\in\GL_n(\Bk)$, we set $M_{\theta'} = M_{\theta} h^{-1}$. Note that this set does not depend on the choice of $h$ and satisfies the same properties for $\theta'$ as $M_{\theta}$ does for $\theta$, i.e. for any $g\in M_{\theta'}$ we have $g \theta' g^{-1}=N_\lambda$, $g$ is kernel-strict, and $M_{\theta'}$ forms a single orbit under the action of $Z(N_{\lambda})\times Z(\theta')$. We have:

\begin{prop}\label{prop:finite weight}
Choose a classification data for $\Bk$. For any nilpotent $\theta\in\Mat_n(\BFR)$ the group $Z(N_{\lambda})$ resp. $Z(\theta)$ acts freely on $M_\theta$ on the left resp. on the right. The sets $Z(N_{\lambda})\backslash M_\theta$, $M_\theta /  Z(\theta)$ are finite.
\end{prop}
\begin{proof}
The first claim is clear. The second claims follows from Proposition \ref{prop:finiteness grass} and the fact that the following natural maps are injective:
\[
Z(N_{\lambda})\backslash M_\theta \hookrightarrow \GL_n(\BFR)\backslash\GL_{n,d}^{\geq -N(\lambda, d)}(\BFK)\cong\AGr_{n,d}^{\geq -N(\lambda, d)}(\Bk),
\]
\[
\qquad M_\theta/ Z(\theta) \hookrightarrow \GL_{n,d}^{\geq -N(\lambda, d)}(\BFK)/GL_n(\BFR)=\AGr_{n,d}^{\geq -N(\lambda, d)}(\Bk).
\]
\end{proof}

The ratio of the two sizes will be important later:
\begin{defn}\label{defn:weight}
Given a classification data over a finite field $\Bk$ the \emph{weight} of a nilpotent matrix $\theta\in\Mat_n(\BFR)$ is defined as the ratio
\[
\weight \theta = \frac{\left| M_\theta /  Z(\theta)\right|}{\left| Z(N_{\lambda})\backslash M_\theta\right|}.
\]
Clearly, it does not depend on the choice of representative of $[\theta]\in\ANilp_n(\Bk)$ and thus we obtain a function $\weight:\ANilp_n(\Bk)\to \BQ_{>0}$.
\end{defn}

\begin{rem}
It will be shown later that $\weight$ does not depend on the choice of a classification data, but we do not have a direct proof of this.
\end{rem}

\begin{rem}\label{rem:weight commensurability}
The weight can be interpreted as the commensurability index $[Z(N_\lambda):g Z(\theta) g^{-1}]$ for any $g\in M_\theta$.
\end{rem}

\section{Modifications of vector bundles}\label{sec:modifications}
\subsection{Affine Grassmanian}
Geometrically, points of the affine Grassmanian $\AGr_n(\Bk)$ parametrize all vector bundles of rank $n$ on the disk $\Spec_{\BFR}$ equipped with a trivialization on the punctured disk $\Spec_{\BFK}$. The points of the \emph{positive part} $\widehat\Gr_n^+(\Bk)$ parametrize all subbundles of rank $n$ of the trivial vector bundle of rank $n$ on the disk $\Spec_{\BFR}$.

\subsection{Extension and restriction of vector bundles}
Here we remind the reader how local modifications to a vector bundle can be described in terms of the affine Grassmanian. 

Let $\Sigma$ be a smooth complete curve over a subfield of $\Bk$, and let $s$ be a closed point of $\Sigma$ with residue field $\Bk$. Let $\CE\in\Bun(\Sigma)$ of rank $n$. Let $\CO_s$ denote the local ring at $s$ and $\widehat{\CO_s}$ the completed local ring at $s$. Since $\Sigma$ is smooth, we can identify $\widehat{\CO_s}$ with $\BFR=\Bk[[x]]$. We can also choose a \emph{trivialization} $\CE\otimes_{\CO_s} \BFR \cong \BFR^n$. Then for any element $[g]\in\widehat\Gr_n^+(\Bk)_L$ the \emph{extension} of $\CE$ is defined as follows. For any open $U\subset \Sigma$ we set
\[
\CE^g(U) = \begin{cases}
\CE(U) & \text{if $s\notin U$,}\\
\{s\in\CE(U\setminus\{s\}):\, g s\in \BFR^n\}  & \text{if $s\in U$.}
\end{cases}
\]
The \emph{restriction} of $\CE$ is defined using $g^{-1}$ for $[g]\in \widehat\Gr_n^+(\Bk)$. For any open $U\subset \Sigma$ we set
\[
\CE_g(U) = \begin{cases}
\CE(U) & \text{if $s\notin U$,}\\
\{s\in\CE(U\setminus\{s\}):\, g^{-1} s\in \BFR^n\}  & \text{if $s\in U$.}
\end{cases}
\]
We have
\begin{prop}
For any vector bundle $\CE$ with a choice of trivialization as above at a closed point $s$ with residue field $\Bk$ and $[g]\in\widehat\Gr_n^+(\Bk)$ resp. $[g]\in\widehat\Gr_n^+(\Bk)_L$ we have that $\CE_g$ resp. $\CE^g$ are vector bundles equipped with natural embeddings
\[
\CE_g\subset \CE\quad\text{resp.}\quad \CE \subset \CE^g,
\]
which are isomorphisms on $\Sigma\setminus\{s\}$.
\end{prop}
\begin{proof}
It is straightforward to check that $\CE_g, \CE^g$ are sheaves. To see that they are vector bundles, we use the fact that $g$ can be represented by a matrix of rational functions on $\Sigma$, which provides a trivialization of $\CE^g$, $\CE_g$ on a neighborhood of $s$.
\end{proof}

Next we show that the constructions of $\CE^g$ and $\CE_g$ are universal in the following sense:
\begin{prop}\label{prop:modifications1}
For $\CE, \CF\in\Bun(\Sigma)$ of rank $n$ and $s$ a closed point of $\Sigma$ with residue field $\Bk$, let $\varphi:\CE\to\CF$ be an embedding such that $\varphi|_{\Sigma\setminus\{s\}}$ is an isomorphism. Then we have
\begin{enumerate}
\item For a choice of trivialization of $\CE$ at $s$ there is a unique $[g]\in \widehat\Gr_n^+(\Bk)_L$ and an isomorphism $\CE^g\xrightarrow{\sim} \CF$ so that $\varphi$ factors as $\CE\hookrightarrow \CE^g \xrightarrow{\sim} \CF$.
\item For a choice of trivialization of $\CF$ at $s$ there is a unique $[g]\in \widehat\Gr_n^+(\Bk)$ and an isomorphism $\CE\xrightarrow{\sim} \CF_g$ so that $\varphi$ factors as $\CE\xrightarrow{\sim} \CF_g \hookrightarrow \CF$.
\end{enumerate}
\end{prop}
\begin{proof}
For a purpose of proof we choose trivializations of both $\CE$ and $\CF$. Note that $\CE^g\otimes \BFR$ comes with a natural trivialization in such a way that the map $\CE\hookrightarrow \CE^g$ is given by the matrix $g$. Similarly, $\CF_g$ has a natural trivialization in such a way that the map $\CF_g\to \CF$ is given by the matrix $g$.

 Tensoring $\varphi$ with $\BFR$ gives a matrix of power series, which we denote by $h$. Let us prove the uniqueness. In statement (i) we see that $h$ factors as $g$ followed by an automorphism. Thus $g$ and $h$ must be in the same orbit with respect to the left multiplication by $\GL_n(\BFR)$, which is equivalent to $[g]=[h]$. In the statement (ii) similarly we obtain $[g]=[h]$. Thus the class of $g$ is unique in both statements. Since $\CF$ is torsion-free and $\CE^g/\CE$ is torsion, the isomorphism in (i) is unique. The isomorphism in (ii) is unique too. So we have proved the uniqueness.
 
 To prove existence we set $g=h$ in (ii). In (i) we construct $\alpha:\CE^g\xrightarrow{\sim}\CF$ as follows. For $U\subset\Sigma$ we have two cases. If $s\notin U$ we set $\alpha(U)=\varphi(U)$. If $s\in U$ note that $\varphi(U\setminus\{s\})$ restricts to a bijection between $\CE^g(U)\subset \CE(U\setminus\{s\})$ and $\CF(U)\subset \CF(U\setminus\{s\})$. So we set $\alpha(U)$ to be this restriction. We proceed similarly in (ii).
\end{proof}

\subsection{Extension and restriction with poles}
If $g$ in the previous subsection has poles, we can still construct extension and restriction of $\CE$, but their relation with $\CE$ is slightly more complicated:
\begin{defn}
Let $\CE$ be a vector bundle of rank $n$ with trivialization at $s$, and let $[g]\in\widehat\Gr_n^+(\Bk)_L$ resp. $[g]\in\widehat\Gr_n^+(\Bk)$ be represented by a matrix $g$ with poles of order at most $N$. Then we define 
\[
\CE^g = \CE(-N)^{x^N g}\quad \text{resp.} \quad \CE_g = \CE(N)_{x^N g},
\]
where $\CE(N)=\CE^{x^N\Id_n}$ resp. $\CE(-N)=\CE_{x^N\Id_n}$ are the positive resp. negative twists of $\CE$ at $s$.
\end{defn}

Proposition \ref{prop:modifications1} generalizes as follows:
\begin{prop}\label{prop:modifications2}
For $\CE, \CF\in\Bun(\Sigma)$ of rank $n$ and $s$ a closed point of $\Sigma$ with residue field $\Bk$, let $\varphi:\CE(-N s)\hookrightarrow \CF$ be an embedding such that $\varphi|_{\Sigma\setminus\{s\}}$ is an isomorphism. Then we have
\begin{enumerate}
\item For a choice of trivialization of $\CE$ at $s$ there is a unique $[g]\in \widehat\Gr_n^{\geq -N}(\Bk)_L$ and an isomorphism $\CE^g\xrightarrow{\sim} \CF$ so that $\varphi$ factors as $\CE(-Ns)\hookrightarrow \CE^g \xrightarrow{\sim} \CF$.
\item For a choice of trivialization of $\CF$ at $s$ there is a unique $[g]\in \widehat\Gr_n^{\geq -N}(\Bk)$ and an isomorphism $\CE\xrightarrow{\sim} \CF_g$ so that $\varphi$ factors as $\CE(-Ns)\xrightarrow{\sim} \CF_g(-Ns) \hookrightarrow \CF$.
\end{enumerate}
\end{prop}

\section{Counting vector bundles with nilpotent endomorphisms}\label{sec:counting vector bundles}
Let $\Sigma$ be a smooth complete curve over $\BF_q$. Although we are mostly interested in the category of coherent sheaves $\Coh(\Sigma)$, some of the statements below make sense for an arbitrary abelian category $\CA$, so we formulate them in this generality.

\subsection{Truncations} When we count vector bundles we need a suitable truncation of the category of vector bundles. The properties that we require of such a truncation are as follows:
\begin{defn}\label{defn:suitable truncation}
Let $\CA$ be an abelian category and let $\tau$ be a property of objects of $\CA$. We call it a \emph{suitable truncation} if it is closed under subobjects and extensions, i.e. for any short exact sequence in $\CA$.
\[
0\to A \to B \to C \to 0
\]
 the following holds:
\begin{enumerate}
\item If $B$ satisfies $\tau$, then $A$ satisfies $\tau$.
\item If $A$ and $C$ both satisfy $\tau$, then $B$ satisfies $\tau$.
\end{enumerate}
\end{defn}

For a suitable truncation $\tau$ on $\CA$ we denote by $\CA^\tau$ the full subcategory of objects of $\CA$ satisfying $\tau$. If $\tau$ is a suitable truncation on $\Coh(\Sigma)$, then we write $\Bun^\tau(\Sigma)$ for the intersection $\Coh^\tau(\Sigma)\cap \Bun(\Sigma)$. Note that the property of being a torsion-free sheaf is itself a suitable truncation, so $\Bun^\tau(\Sigma)=\Coh^{\tau'}(\Sigma)$ where $\tau'$ means ``satisfies $\tau$ and is torsion-free''. A basic example of $\tau$ is $\leq 0$, where we say that $A\in\Coh(\Sigma)$ satisfies $\leq 0$ if and only if for every subobject $B\subset A$ we have $\deg B\leq 0$. Note that this implies torsion-free. Note that our notion of truncation is dual to Schiffmann's. Shiffmann's truncation $\geq 0$ is closed under extensions and quotients. However, the property of being a vector bundle is not closed under quotients, so our notion is a little more convenient. We collect here a couple of useful properties:
\begin{prop}
If $A=B\oplus C$ in $\CA$, then $A$ satisfies $\tau$ if and only if both $B$ and $C$ satisfy $\tau$.
\end{prop} 
\begin{proof}
Straightforward.
\end{proof}

\begin{prop}\label{prop:truncation nilpotent}
Let $A\in \CA$, and let $\theta:A\to A$ be nilpotent. Then $A$ satisfies $\tau$ if and only if $\kernel\theta$ satisfies $\tau$.
\end{prop}
\begin{proof}
If $A$ satisfies $\tau$, then $\kernel \theta$ is a subobject of $A$, so it satisfies $\tau$. To prove the other implication, note that we have a short exact sequence for each $k>0$
\begin{equation}
0\to \kernel \theta^k \to \kernel\theta^{k+1} \to \image \theta^k|_{\kernel\theta^{k+1}} \to 0.
\end{equation}
Here $\image \theta^k|_{\kernel\theta^{k+1}}\subset \kernel\theta$, so we can inductively deduce that $\kernel\theta^{k+1}$ satisfies $\tau$ for all $k$.
\end{proof}

\subsection{Counting}
Let $\Sigma$ be a smooth complete curve over $\BF_q$ and $\tau$ a suitable truncation. Since we want to count bundles, we assume the following
\begin{defn}
A suitable truncation $\tau$ is locally finite if the number of isomorphism classes of bundles of rank $n$ and degree $\geq d$ satisfying $\tau$ is finite for all $n$ and $d$.
\end{defn}
For instance, $\tau =$ ``$\leq 0$'' satisfies this property. Denote by $\Bun^\tau_\nil(\Sigma)$ the category of pairs $(\CE, \theta)$ where $\CE\in\Bun^\tau(\Sigma)$ and $\theta:\CE\to \CE$ nilpotent.

Let the rank of $\CE$ be $n$.
The \emph{global type} $\type(\theta)$ is defined to be the partition of size $n$ with entries $\type(\theta)_i = \rank \kernel \theta^i - \rank \kernel \theta^{i-1}$, $i=1,2,\ldots$. This coincides with the type defined earlier for $\theta$ viewed as a matrix over the field of functions on $\Sigma$.  Denote by $\Bun^\tau_{\lambda}(\Sigma)$ resp. $\Bun^\tau_{\nil,n}(\Sigma)$ the full subcategory of $\Bun^\tau_{\nil}(\Sigma)$ of pairs $(\CE, \theta)$ such that $\type \theta = \lambda$ resp. $\rank\CE=n$. For $s$ a closed point of $\Sigma$ we also have $\type \theta(s)$ as the type of $\theta$ restricted to the fiber of $\CE$ over $s$. We introduce the following counting functions
\[
\Omega_{\lambda}^\tau(t) = \sum_{(\CE,\theta)\in\Bun^\tau_{\lambda}(\Sigma)/\sim} \frac{t^{-\deg \CE}}{|\Aut(\CE, \theta)|} \in \BQ((t)),
\]
\[
\Omega_{n}^\tau(t) = \sum_{(\CE,\theta)\in\Bun^\tau_{\nil,n}(\Sigma)/\sim} \frac{t^{-\deg \CE}}{|\Aut(\CE, \theta)|} = \sum_{\lambda \vdash n} \Omega_{\lambda}^\tau(t).
\]
Let $S=(s_1,s_2,\ldots,s_k)$ be a collection of distinct closed points of $\Sigma$ of degrees $d_1,d_2,\ldots,d_k$.  We can incorporate the types of $\theta(s_i)$ into our count in the following way:
\[
\Omega_{\lambda,S}^\tau[X_\bullet;t] = \sum_{(\CE,\theta)\in\Bun^\tau_\lambda(\Sigma)/\sim} \frac{t^{-\deg \CE}}{|\Aut(\CE, \theta)|} \prod_{i=1}^k H_{\type \theta(s_i)}[X_i;q^{d_i}],
\]
and 
\[
\Omega_{n,S}^\tau[X_\bullet;t] = \sum_{\lambda \vdash n} \Omega_{\lambda,S}^\tau[X_\bullet;t].
\]
The result is a symmetric function in $k$ infinite groups of variables $X_1, X_2,\ldots, X_k$ where $X_i=(x_{i,1}, x_{i,2},\ldots)$ with coefficients in $\BQ((t))$. Here we use a version of Hall-Littlewood polynomials obtained by $t=0$ specialization of the modified Macdonald polynomials (see Subsection \ref{ssec:macdonald}),
\[
H_\lambda[X;q]=\tilde H_\lambda[X;q,0].
\]

By Corollary \ref{cor:hall counting flags}, the Hall-Littlewood polynomials have the following interpretation. If $M$ is a nilpotent matrix over $\BF_q$ of type $\lambda$, then we have
\[
H_\lambda[X;q] = \sum_{\mu\vdash n} m_\mu[X]\; |\{\text{partial flags of type $\mu$ preserved by $M$}\}|.
\]

Thus we obtain another interpretation of $\Omega_{n,\lambda,S}^\tau(t)$ as follows. Let $\Bmu=(\mu^{(1)}, \ldots, \mu^{(k)})$ be a collection of partitions, one for each marked point. A \emph{parabolic bundle} of type $\Bmu$ is a pair $(\CE, \Bf)$ of a bundle $\CE$ and a collection $\Bf=(f_1,\ldots, f_k)$ where $f_i$ is a partial flag of type $\mu^{(i)}$ in the fiber $\CE(s_i)$ for each $i$. Denote the category of parabolic bundles of type $\Bmu$ by $\Bun(\Sigma;S,\Bmu)$. Denote by $\Bun^\tau_\nil(\Sigma;S,\Bmu)$ the category of parabolic bundles with nilpotent endomorphism $(\CE, \Bf, \theta)$ such that $\CE$ satisfies $\tau$. Then we have
\begin{equation}\label{eq:omega with flags}
\Omega_{\lambda,S}^\tau[X_\bullet;t] = \sum_{\Bmu=(\mu^{(1)}, \ldots, \mu^{(k)})} \prod_{i=1}^k m_{\mu^{(i)}}[X_i] \sum_{\substack{(\CE, \Bf, \theta) \in \Bun^\tau_\nil(\Sigma;S,\Bmu)/\sim \\ \type \theta=\lambda}} \frac{t^{-\deg \CE}}{|\Aut(\CE, \Bf, \theta)|}.
\end{equation}

The main result of this section is to show
\begin{thm}\label{thm: with macdonald}
For a smooth complete curve $\Sigma/\BF_q$, a locally finite suitable truncation $\tau$, a collection of $k$ closed points $S\subset \Sigma$ of degrees $d_1,\ldots,d_k$ and a partition $\lambda\vdash n$, we have the following factorization of the weighted number of parabolic bundles of rank $n$ satisfying $\tau$ with a nilpotent endomorphism of type $\lambda$:
\[
\Omega_{\lambda,S}^\tau[X_\bullet;t] = \Omega_{\lambda}^\tau(t) \prod_{i=1}^k \tilde H_\lambda[X_i;q^{d_i}, t^{d_i}],
\]
where $\Omega_{\lambda}^\tau(t)$ is the corresponding number without parabolic structure.
\end{thm}
The following subsections are devoted to a proof of this theorem.

\subsection{Existence of a factorization}
We first show that the theorem holds for some unknown functions in place of the modified Macdonald polynomials. This subsection is devoted to a proof of the following, where we first assume a choice of a classification data for all finite fields, and then show that the result is independent of the choice:
\begin{thm}\label{thm:existence of factorization}
For any smooth complete curve $\Sigma$ over $\BF_q$, a locally finite suitable truncation $\tau$, a partition $\lambda$, a sequence of points $S$ we have
\[
\Omega_{\lambda,S}^\tau[X_\bullet;t] = \Omega_{\lambda}^\tau(t) \prod_{i=1}^k F_{\lambda,q^{d_i}}[X_i;t^{d_i}],
\]
where $F_{\lambda,q}[X;t] = \frac{C_{\lambda,q}[X;t]}{C_{\lambda,q}[1;t]}$ and
\[
C_{\lambda, q}[X;t] = \sum_{[\eta]\in\ANilp_\lambda(\BF_q)} t^{\deg\eta} \weight(\eta) H_{\type \eta(0)}[X;q].
\]
\end{thm}
\begin{proof}
We first assume a choice of a classification data as in Subsection \ref{ssec:classification data} for all finite fields $\BF_q$. 
We will proceed by induction on $k$. Let $S'=(s_1,\ldots, s_{k-1})$ and $s=s_k$ with $\deg s = d$. Let $\Bk$ be the residue field of $s$. We can choose an identification of $\Bk$ with $\BF_{q^d}$ and therefore a classification data for $\Bk$. 

Consider any $(\CE, \theta)\in\Bun^\tau_\nil(\Sigma)$ of type $\lambda$. We first choose a trivialization at $s$ so that $\theta\otimes\hat\CO_s$ gives a nilpotent matrix, which we denote by $\theta_s$ and write
\[
1 = \sum_{[g]\in Z(N_{\lambda})\backslash M_{\theta_s}} \frac{1}{|Z(N_{\lambda})\backslash M_{\theta_s}|},
\]
where the quotient is finite by Proposition \ref{prop:finite weight}.
Note that we have a natural bijection $Z(N_{\lambda})\backslash M_{\theta_s}\cong \GL_n(\BFR)\backslash \GL_n(\BFR) M_{\theta_s}$, so we can write
\[
1 = \sum_{[g]\in \GL_n(\BFR)\backslash \GL_n(\BFR) M_{\theta_s}} \frac{1}{|Z(N_{\lambda})\backslash M_{\theta_s}|}.
\]
Each $[g]$ above is a point on $\AGr_n(\Bk)_L$, so we can apply the construction of bundle $\CE^{g}$. Since $g\theta_s g^{-1}\in\Mat_n(\BFR)$, we have that $\theta$ uniquely extends to $\theta^{g}$ on $\CE^{g}$, and since $g$ is kernel-strict, we have that $\kernel \theta^{g}$ is isomorphic to $\kernel \theta$. By Proposition \ref{prop:truncation nilpotent} this implies that $\CE$ satisfies $\tau$ if and only if $\CE^{g}$ does. Moreover, we have $\type \theta^{g}=\lambda$ and $\type \theta^{g}(s_i)=\type \theta(s_i)$ for all $i<k$. Let us denote the full subcategory of $\Bun^\tau_\lambda(\Sigma)$ consisting of pairs $(\CF, \theta')$ such that $\type \theta'(s)=\lambda$ by $\Bun'^\tau_\lambda(\Sigma)$. We then have
\[
1 = \sum_{[g]\in \GL_n(\BFR)\backslash \GL_n(\BFR) M_{\theta_s}} \frac{1}{|Z(N_{\lambda})\backslash M_{\theta_s}|} \sum_{\substack{(\CF, \theta')\in \Bun'^\tau_{\lambda}/\sim\\ \varphi:(\CE^{g},\theta^{g})\toiso (\CF,\theta')}}\frac{1}{|\Aut(\CF, \theta')|}.
\]
\[
=\sum_{(\CF, \theta')\in \Bun'^\tau_{\lambda}/\sim} \frac{1}{|\Aut(\CF, \theta')|} \sum_{\substack{[g]\in \GL_n(\BFR)\backslash \GL_n(\BFR) M_{\theta_s} \\ \varphi:(\CE^g,\theta^g)\toiso (\CF,\theta')}} \frac{1}{|Z(N_{\lambda})\backslash M_{\theta_s}|}.
\]

Let us choose a trivialization of $\CF$ at $s$ for each $(\CF, \theta')$ in the sum above such that the matrix of $\theta'\otimes\widehat\CO_s$ is $N_\lambda$. 
By part (i) of Proposition \ref{prop:modifications2} the pairs $([g], \varphi)$ in the above summation are in bijection with the corresponding subset of embeddings $\CE(-Ns)\hookrightarrow \CF$. The relevant subset can be described as those embeddings whose matrix at $s$ is in $M_{\theta_s}$. By part (ii) of the same proposition we can identify it with the corresponding subset of pairs $g, \psi$ where $\psi:\CE\toiso \CF_g$ and $[g]\in M_{\theta_s} \GL_n(\BFR)/ \GL_n(\BFR)$. Thus we obtain
\[
1 = \sum_{(\CF, \theta')\in \Bun'^\tau_{\lambda}/\sim} \frac{1}{|\Aut(\CF, \theta')|} \sum_{\substack{[g]\in M_{\theta_s} \GL_n(\BFR)/ \GL_n(\BFR) \\ \psi:(\CE,\theta)\toiso (\CF_g,\theta'_g)}} \frac{1}{|Z(N_{\lambda})\backslash M_{\theta_s}|}.
\]
Now summing over all $(\CE, \theta)$ and using the equality $\deg \CE = \deg \CF - d \deg\theta_s$ gives
\[
\Omega_{n,\lambda,S}^\tau[X_\bullet;t] = \sum_{\substack{(\CF, \theta')\in \Bun'^\tau_{\lambda}/\sim}} \frac{t^{-\deg \CF}}{|\Aut(\CF, \theta')|} \prod_{i=1}^{k-1} H_{\type \theta'(s_i)}[X_i;q^{d_i}] C_{\CF,\theta'}[X_k;t],
\]
where
\[
C_{\CF, \theta'}[X;t]=\sum_{(\CE,\theta)\in\Bun^\tau_{\lambda}(\Sigma)/\sim} \frac{t^{d \deg \theta_s}}{|\Aut(\CE, \theta)|} \sum_{\substack{[g]\in M_{\theta_s} \GL_n(\BFR)/ \GL_n(\BFR) \\ \psi:(\CE,\theta)\toiso (\CF_g,\theta'_g)}} \frac{H_{\type \theta(s)}[X;q^{d}]}{|Z(N_{\lambda})\backslash M_{\theta_s}|}
\]
\[
= \sum_{\substack{[\eta]\in\ANilp_\lambda(\Bk)\\ [g]\in M_{\eta} \GL_n(\BFR)/ \GL_n(\BFR)}} \frac{t^{d \deg\eta} H_{\type \eta(0)}[X;q^{d}]}{|Z(N_{\lambda})\backslash M_{\eta}|} 
\sum_{\substack{(\CE,\theta)\in\Bun^\tau_{\lambda}(\Sigma)/\sim \\ \psi:(\CE,\theta)\toiso (\CF_g,\theta'_g)}} \frac{1}{|\Aut(\CE, \theta)|}.
\]
Since $(\CF_g,\theta_g')$ satisfies $\tau$ for all $g$ above, the second summation always produces $1$. So we obtain 
\[
C_{\CF, \theta'}[X;t] = \sum_{\substack{[\eta]\in\ANilp_\lambda(\Bk)\\ [g]\in M_{\eta} \GL_n(\BFR)/ \GL_n(\BFR)}} \frac{t^{d \deg\eta} H_{\type \eta(0)}[X;q^{d}]}{|Z(N_{\lambda})\backslash M_{\eta}|}.
\]
For each $[\eta]$ the set $M_{\eta} \GL_n(\BFR)/ \GL_n(\BFR)$ is in bijection with $M_{\eta}/Z(\eta)$, which is finite by Proposition \ref{prop:finite weight}. Thus we finally arrive at 
\[
C_{\CF, \theta'}[X;t] = \sum_{[\eta]\in\ANilp_\lambda(\Bk)} t^{d \deg\eta} \weight(\eta) H_{\type \eta(0)}[X;q^{d}],
\]
which does not depend on $\CF, \theta'$. So we have
\[
C_{\CF, \theta'}[X;t]= C_{\lambda,q^{d}}[X;t^{d}],
\]
where
\[
C_{\lambda, q}[X;t] = \sum_{[\eta]\in\ANilp_\lambda(\BF_q)} t^{\deg\eta} \weight(\eta) H_{\type \eta(0)}[X;q].
\]
Write
\[
C_{\lambda,q}[1;t] = \sum_{[\eta]\in\ANilp_\lambda(\Bk)} t^{\deg\eta} \weight(\eta).
\]
We have that $C_{\lambda,q}[1]\in 1 + t\BQ[[t]]$ because by Corollary \ref{cor:non-degenerate} the only degree $0$ class is $N_\lambda$, and its weight is clearly $1$. From the following counts
\[
\Omega_{\lambda,S}^\tau[X_1,\ldots,X_k;t] = \Omega'^\tau_{\lambda,S}[X_1,\ldots,X_{k-1};t] C_{\lambda,q^d}[X_k;t^d],
\]
\[
\Omega_{\lambda,S'}^\tau[X_1,\ldots,X_{k-1};t] = \Omega'^\tau_{\lambda,S}[X_1,\ldots,X_{k-1};t] C_{\lambda,q^d}[1;t^d],
\]
where
\[
\Omega'^\tau_{\lambda,S}[X_1,\ldots,X_{k-1};t] = \sum_{\substack{(\CF, \theta')\in \Bun'^\tau_{\lambda}/\sim}} \frac{t^{-\deg \CF}}{|\Aut(\CF, \theta')|} \prod_{i=1}^{k-1} H_{\type \theta'(s_i)}[X_i;q^{d_i}],
\]
we obtain
\[
\Omega_{\lambda,S}^\tau[X_1,\ldots,X_{k};t] = \Omega_{\lambda,S'}^\tau[X_1,\ldots,X_{k-1};t] \frac{C_{\lambda,q^d}[X;t^d]}{C_{\lambda,q^d}[1;t^d]}.
\]
So the desired factorization holds.
\end{proof}

\begin{cor}[Of the proof]
We have that $\weight \theta$ for a nilpotent matrix $\theta\in\Mat_n(\BFR)$ over a finite field $\Bk$ does not depend on the choice of classification data.
\end{cor}
\begin{proof}
Let $\Sigma=\BP^1/\Bk$, $k=1$, $s=0$, $\tau=$ ``$\leq 0$''. Let $[\eta]\in \ANilp_n(\Bk)$, $\type \eta = \lambda$ $\deg \eta=d$. Note that the only bundle of degree $0$ satisfying $\tau$ is the trivial bundle. All endomorphisms of the trivial bundle are constant, so all pairs $(\CF, \theta')$ of degree $0$ are equivalent to the pair $(\CO_{\BP^1}^n, N_\lambda)$. By following the proof of Theorem \ref{thm:existence of factorization} we obtain
\[
\sum_{\substack{(\CE, \theta)\in \Bun^{\leq 0}_{\lambda}(\BP^1)/\sim\\ \deg\CE=-d,\, [\theta_0]=[\eta]}}\frac{1}{|\Aut(\CE,\theta)|} = \frac{\weight \eta}{|\Aut(\CO_{\BP^1}^n, N_\lambda)|} = \frac{\weight\eta}{z_\lambda(q)},
\]
where $z_\lambda(q)$ is the size of the centralizer of $N_\lambda$ in $\GL_n(\Bk)$. Thus we have
\[
\weight \eta = z_\lambda(q) \sum_{\substack{(\CE, \theta)\in \Bun^{\leq 0}_{\lambda}(\BP^1)/\sim\\ \deg\CE=-d,\, [\theta_0]=[\eta]}}\frac{1}{|\Aut(\CE,\theta)|},
\]
which indeed depends only on $\Bk, \lambda$ and $[\eta]$, but not on the choice of the classification data.
\end{proof}

\subsection{Computation on $\BP^1$ with $2$ marked points}\label{subs: P1}
In this subsection we compute $\Omega_{n,(0,\infty)}^{\leq 0}(\BP^1)[X,Y;t]$ over $\Bk=\BF_q$,
i.e. the weighted numbers of vector bundles on $\BP^1$ with no positive degree subbundles with nilpotent endomorphism and parabolic structures at $0$ and $\infty$. A vector bundle over $\BP^1$ is a direct sum of line bundles. Let the multiplicities of the line bundles be given by a composition $\mu_1+\ldots+\mu_m=n$ with $\mu_i>0$ and the negative degrees of the line bundles by numbers $0\leq d_1\leq d_2\leq\ldots\leq d_m$, so that 
\[
\CE = \CO(-d_1)^\mu_1 \oplus\cdots\oplus \CO(-d_m)^\mu_m.
\] 
We will decompose matrices below into blocks of sizes $\mu_1,\ldots,\mu_m$ so that an $n\times n$ matrix $M$ corresponds to matrices $M_{i,j}$ with $1\leq i,j\leq m$ of shapes $\mu_i\times\mu_j$.
The automorphisms of $\CE$ are block-upper-triangular. In the diagonal blocks we have arbitrary invertible constant matrices. In the off-diagonal block $i,j$ we have arbitrary $\mu_i\times\mu_j$ matrix of polynomials of degree $\leq d_j-d_i$. Thus we have
\[
|\Aut(\CE)| = \prod_{i=1}^m |\GL_{\mu_i}(\Bk)| \prod_{i<j} q^{\mu_i\mu_j(d_j-d_i+1)}.
\]
A nilpotent endomorphism $\theta$ of $\CE$ is again given by block-upper-triangular matrix with block sizes $\mu_1,\ldots,\mu_m$. In the diagonal blocks we have arbitrary nilpotent constant matrices. In the off-diagonal block $i,j$ we have arbitrary $\mu_i\times\mu_j$ matrix of polynomials of degree $\leq d_j-d_i$. If we specify $A=\theta(0)$ and $B=\theta(\infty)$ then we specify the diagonal blocks, which must coincide, and the highest and the lowest coefficient of each of the off-diagonal block. Thus the number of nilpotent endomorphisms with given $A$ and $B$ is $\prod_{i<j} q^{\mu_i\mu_j(d_j-d_i-1)}$. Let $Q_\mu$ be the set of pairs of nilpotent block-upper-triangular matrices $A$, $B$ satsifying $A_{i,i}=B_{i,i}$ for all $i$. The total contribution of $\CE$ to $\Omega_{n,(0,\infty)}^{\leq 0}(\BP^1)[X,Y;t]$ is given by the following expression: 
\[
t^{\sum_i d_i \mu_i} C_\mu[X,Y],
\]
where
\[
C_\mu[X,Y]=
\frac{\sum_{(A,B)\in Q_\mu} H_{\type A}[X;q] H_{\type B}[Y;q]}{\prod_{i=1}^m |\GL_{\mu_i}(\Bk)| \prod_{i<j} q^{2\mu_i\mu_j}}.
\]
Notice that $|\GL_{\mu_i}(\Bk)| \prod_{i<j} q^{\mu_i\mu_j}=|P_\mu|$ where $P_\mu$ is the parabolic subgroup of $\GL_n(\Bk)$ consisting of block-upper triangular matrices. We have that $P_\mu$ is the stabilizer in $\GL_n(\Bk)$ of the standard flag of type $\mu$. Therefore we can rewrite $C_\mu[X,Y]$ as follows:
\[
C_\mu[X,Y] = \frac{\sum_{(A,B,F)\in R_\mu} H_{\type A}[X;q] H_{\type B}[Y;q]}{|\GL_n(\Bk)| \prod_{i<j} q^{\mu_i\mu_j}},
\]
where $R_\mu$ is the set of triples $(A,B,F)$ such that $A,B\in\Mat_n(\Bk)$ are nilpotent, $F$ is a flag of type $\mu$, $A,B$ preserve $F$ and the actions of $A$ and $B$ on the associated graded space with respect to $F$ coincide. All such triples can be obtained as follows. We split the sum according to the types of $A$ and $B$. Let $\type A = \lambda$, $\type B = \nu$. Then $A = g N_\lambda g^{-1}$, and the number of $g$ giving the same $A$ equals $z_\lambda(q)$, which is the centralizer of $N_\lambda$ in $\GL_n(\Bk)$. Similarly we write $B=h N_\mu h^{-1}$. Let $F_{\lambda,\mu}$ denote the set of flags of type $\mu$ preserved by $N_\lambda$. Then $C_\mu[X,Y]$ is rewritten as follows:
\[
C_\mu[X,Y] = \sum_{\lambda,\nu\vdash n} C_{\mu,\lambda,\nu} H_{\lambda}[X;q] H_{\nu}[Y;q],
\] 
where 
\[
C_{\mu,\lambda,\nu} = \frac{\left|\{F_1\in F_{\lambda,\mu},\, F_2\in F_{\nu,\mu},\,g_1,g_2\in\GL_n(\Bk):\,g_1 F_1 = g_2 F_2,\,*\}\right| }{z_\lambda(q) z_\nu(q) |\GL_n(\Bk)| \prod_{i<j} q^{\mu_i\mu_j}}.
\]
Here $*$ means the following condition: the actions induced by $g_1 N_\lambda g_1^{-1}$ and $g_2 N_\nu g_2^{-1}$ on the associated graded w.r.t. $g_1 F_1=g_2 F_2$ must coincide. This is equivalent to the condition that the associated graded of $N_\lambda$ on $F_1$ must coincide with the associated graded of $g_1^{-1} g_2 N_\nu g_2^{-1} g_1$ on $F_1$. Let $g=g_1^{-1} g_2$. Clearly, the number of pairs $g_1, g_2$ producing the same $g$ equals $|\GL_n(\Bk)|$. This cancels $|\GL_n(\Bk)|$ in the denominator. The condition on $g$ is that it must send $F_2$ to $F_1$, and on the associated graded it must conjugate the action of $N_\nu$ to that of $N_\lambda$. The number of $g$ with the same associated graded equals $\prod_{i<j} q^{\mu_i\mu_j}$, which cancels out with the corresponding product in the denominator. For each $F_1\in F_{\lambda,\mu}$ denote by $\kappa_i(F_1)=\type N_\lambda|_{(F_1)_i/(F_1)_{i-1}}$ and by $\kappa(F_1)$ the sequence of partitions $(\kappa_1(F_1),\ldots, \kappa_m(F_m))$. Similarly define $\kappa(F_2)$.
\[
C_{\mu,\lambda,\nu} = \frac{\sum_{\substack{F_1\in F_{\lambda,\mu},\, F_2\in F_{\nu,\mu} \\ \kappa(F_1)=\kappa(F_2)}} \prod_{i=1}^m z_{\kappa(F_1)_i}(q)}{z_\lambda(q) z_\nu(q)}.
\]
We will use the explicit description of the Hall algebra of nilpotent matrices over a field given in Theorem \ref{thm:hall algebra}. The $q$-deformed scalar product $(H_\lambda[Z;q], H_\nu[Z;q])_q$ precisely equals $z_\lambda(q)$ if $\lambda=\nu$ and vanishes otherwise.
So we have
\[
C_{\mu,\lambda,\nu} = \frac{\sum_{F_1\in F_{\lambda,\mu},\, F_2\in F_{\nu,\mu}} \prod_{i=1}^m (H_{\kappa(F_1)_i}[Z;q], H_{\kappa(F_2)_i}[Z;q])_q}{z_\lambda(q) z_\nu(q)}
\]
\[
= \frac{\left(\sum_{F_1\in F_{\lambda,\mu}} \bigotimes_{i=1}^m H_{\kappa(F_1)_i}[Z;q],\; \sum_{F_2\in F_{\nu,\mu}} \bigotimes_{i=1}^m H_{\kappa(F_2)_i}[Z;q]\right)_q}{z_\lambda(q) z_\nu(q)}.
\]
In particular, we see that $C_{\mu,\lambda,\nu}=C_{\mu,\lambda,\nu}(q)$ is a rational function of $q$.
Note that $\sum_{F_1\in F_{\lambda,\mu}} \bigotimes_{i=1}^m H_{\kappa(F_1)_i}[Z;q]$ is precisely the $\mu_1,\mu_2,\ldots,\mu_m$-degree part of the $m-1$-st iterated coproduct in the Hall algebra applied to $H_\lambda[Z;q]$ and similarly for the term with $\nu$. Note also that by Cauchy formula
\[
\sum_{\lambda\in\CP} \frac{H_\lambda[Z;q] H_\lambda[X;q]}{z_\lambda(q)} = \pExp\left[\frac{ZX}{q-1}\right] = \sum_{\lambda\in\CP} h_\lambda\left[\frac{X}{q-1}\right] m_\lambda[Z],
\]
\[
\sum_{\nu\in\CP} \frac{H_\lambda[Y;q] H_\lambda[Z;q]}{z_\lambda(q)} =\sum_{\nu\in\CP} h_\lambda\left[\frac{Z}{q-1}\right] m_\lambda[Y],
\]
and both elements are group-like. Thus we obtain
\[
C_\mu[X,Y;q] = \prod_{i=1}^m \left(\sum_{\lambda\vdash \mu_i} h_\lambda\left[\frac{X}{q-1}\right] m_\lambda[Z], \sum_{\nu\vdash \nu_i} h_\lambda\left[\frac{Z}{q-1}\right] m_\lambda[Y]\right)_q
\]
\[
=\prod_{i=1}^m \sum_{\lambda\vdash \mu_i} h_\lambda\left[\frac{X}{q-1}\right] m_\lambda[Y] = \prod_{i=1}^m h_{\mu_i}\left[\frac{XY}{q-1}\right].
\]
Now we can finish the calculation:
\[
\sum_{n=0}^\infty \Omega_{n,(0,\infty)}^{\leq 0}(\BP^1)[X,Y;t] = \sum_{d,\mu} t^{\sum_i d_i \mu_i} C_\mu[X,Y;q] = \prod_{d=0}^\infty \sum_{k=0}^\infty t^{d k} h_{k}\left[\frac{XY}{q-1}\right]
\]
\[
=\prod_{d=0}^\infty \pExp\left[\frac{t^d XY}{q-1}\right] = \pExp\left[\frac{XY}{q-1}\sum_{d=0}^\infty t^d\right] = \pExp\left[\frac{XY}{(q-1)(1-t)}\right].
\]

\subsection{Upper-triangularity}
We are ready to identify the unknown functions $F_{\lambda,q}[X;t]$ with Macdonald polynomials. Note that we do not even know that the dependence on $q$ is given by a rational function.
\begin{prop}\label{prop:F upper triangular}
For any partition $\lambda$ we have $F_{\lambda,q}[1;t]=1$ and $F_{\lambda,q}[(q-1)X;t]\in M_{\preceq \lambda'}$.
\end{prop}
\begin{proof}
The first claim is clear from $F_{\lambda,q}[X;t] = \frac{C_{\lambda,q}[X;t]}{C_{\lambda,q}[1;t]}$. To show the second claim notice that $\type\theta(0) \succeq \type \theta$ for any $[\theta]\in\ANilp(\BF_q)$, which follows from the fact that $\dim\kernel\theta^i(s)$ is an upper-semicontinuous function of $s$ for all $i$. Thus we see that $C_{\lambda,q}$ is a linear combination of Hall-Littlewood polynomials $H_\nu[X;q]$ with $\nu\succeq\lambda$, or equivalently $\nu'\preceq \lambda'$. By Definition \ref{defn:macdonald} we have for each such $\nu$
\[
H_\nu[(q-1)X;q] \in M_{\preceq \nu'} \subset M_{\preceq \lambda'}.
\]
\end{proof}

Now we are ready to complete the proof
\begin{proof}[Proof of Theorem \ref{thm: with macdonald}]
By the computation in Subsection \ref{subs: P1} and Theorem \ref{thm:existence of factorization}, in the case of $\BP^1$ we have
\[
\pExp\left[\frac{XY}{(q-1)(1-t)}\right] = \sum_{\lambda\in\CP} \Omega^{\leq 0}_{\lambda}(t) F_{\lambda,q}[X;t] F_{\lambda,q}[Y;t].
\]
In particular we see that the operator $\Sym[Z]\to\Sym[Z]$ defined by
\[
G\to \sum_{\lambda\in\CP} \Omega^{\leq 0}_{\lambda}(t) F_{\lambda,q}[Z;t] ( F_{\lambda,q}[X;t], G[X])_{q,t}
\]
is the identity operator. Therefore $F_{\lambda,q}$ span the space of symmetric functions. Since the number of $F_{\lambda,q}$ of degree $d$ equals to the number of partitions, we conclude that $F_{\lambda,q}$ form a basis of $\Sym[Z]$ over the field of Laurent series in $t$. We substitute $F_\mu[X;t]$ into the operator above and obtain
\[
(F_{\lambda,q}[X;t], F_{\mu,q}[X;t])_{q,t} = \delta_{\lambda,\mu} \frac{1}{\Omega^{\leq 0}_{\lambda}(t)}.
\]
By Proposition \ref{prop:F upper triangular} and Corollary \ref{cor:unique macdonald} we conclude that for all partitions $\lambda$
\[
F_{\lambda,q} = \tilde H_{\lambda,q}.
\]
Notice that we have also obtained a proof of the following identity for $\Sigma=\BP^1$:
\[
\Omega^{\leq 0}_{\lambda}(t) = \frac{1}{(\tilde H_\lambda, \tilde H_\lambda)_{q,t}}.
\]
\end{proof}

\begin{cor}[Of the proof]\label{cor:genus 0}
When $\Sigma=\BP^1/\BF_q$ and $S=(s_1,\ldots,s_k)$ a collection of points of degrees $d_1,\ldots,d_k$ we have 
\[
\Omega_{\lambda,S}^{\leq 0} [X_\bullet;t] = \frac{\prod_{i=1}^k \tilde H_\lambda[X_i;q^{d_i}, t^{d_i}]}{(\tilde H_\lambda, \tilde H_\lambda)_{q,t}}
\]
\end{cor}

When $g>0$ we use a result of Schiffmann \cite{schiffmann2014indecomposable}, which says that there are exist explicit functions $\Omega_{g,\lambda}\in\BQ(q,t)[\sigma_1,\ldots,\sigma_{2g}]$ such that for any smooth complete curve $\Sigma/\BF_q$ of genus $g$ with zeta function
\[
\zeta_\Sigma(T) = \frac{\prod_{i=1}^{g}(1-\sigma_i T)(1-q\sigma_i^{-1}T)}{(1-T)(1-qT)},
\]
we have
\[
\Omega_{\lambda}^{\leq 0}(t) =\Omega_{g,\lambda} (q,t,\sigma_1,\ldots,\sigma_{2g}). 
\]
We denote by $\sigma$ the collection $\sigma=(\sigma_1,\ldots,\sigma_g)$.
\begin{cor}\label{cor:genus g}
Let $g\geq 0$. For any smooth complete curve $\Sigma/\BF_q$ and $S=(s_1,\ldots,s_k)$ a collection of points of degrees $d_1,\ldots,d_k$ we have 
\[
\Omega_{\lambda,S}^{\leq 0} [X_\bullet;t] = \Omega_{g,\lambda}(q,t,\sigma) \prod_{i=1}^k \tilde H_\lambda[X_i;q^{d_i}, t^{d_i}],
\]
where $\Omega_{g,\lambda}$ and $\sigma_i$ are defined above.
\end{cor}

\subsection{Further results}
Comparing Theorems \ref{thm:existence of factorization} and \ref{thm: with macdonald} we obtain a following formula for the weighted count of nilpotent matrices over $\BF_q[[t]]$. Recall
\[
C_{\lambda, q}[X;t] = \sum_{[\eta]\in\ANilp_\lambda(\BF_q)} t^{\deg\eta} \weight(\eta) H_{\type \eta(0)}[X;q],
\]
and we have
\[
\frac{C_{\lambda, q}[X;t]}{C_{\lambda, q}[1;t]} = \tilde H_\lambda[X;q,t].
\]
We can determine the function $C_{\lambda, q}[1;t]$:
\begin{prop}
The function $C_{\lambda, q}[1;t]$ is given by
\[
C_{\lambda, q}[1;t] = \prod_{\substack{a,l\\l\neq 0}}\frac{1}{1-t^l q^{-a-1}},
\]
where the product is over the arm- and leg-lengths $a,l$ of the hooks of $\lambda$ with $l\neq 0$.
\end{prop}
\begin{proof}
By Corollary \ref{cor:non-degenerate} we have that $C_{\lambda, q}[X;t]$ has a unique term with $\type\eta(0)=\lambda$, and this term has coefficient $1$. Thus we have
\begin{equation}\label{eq:C to H}
C_{\lambda, q}[X;t] = H_\lambda[X;q] + \langle H_\nu[X;q]: \nu\succeq\lambda\rangle,
\end{equation}
where by $\langle \rangle$ we denote the linear span of a given set of elements. Recall (\cite{garsia1996remarkable})
\[
\tilde H_\lambda[X;q,t]=t^{n(\lambda)} J_{\lambda}\left[\frac{X}{1-t^{-1}};q,t^{-1}\right] = t^{n(\lambda)+|\lambda|} J_{\lambda}\left[\frac{X}{t-1};q,t^{-1}\right],
\]
where $J_\lambda$ is the \emph{integral form Macdonald polynomial} \cite{macdonald1995symmetric} and $n(\lambda)$ is the sum of the leg lengths of hooks of $\lambda$. The integral form is defined by 
\[
J_\lambda[X;q,t] = \prod_{a,l}(1-q^a t^{l+1}) P_\lambda[X;q,t]
\]
where 
\[
P_\lambda[X;q,t] = m_\lambda[X;q,t]+\langle m_\nu: \nu\preceq\lambda\rangle.
\]
Thus we have
\[
\tilde H_\lambda[(t-1) X;q,t] = \prod_{a,l}(t^{l+1}-q^a) m_\lambda + \langle m_\nu: \nu\preceq\lambda\rangle.
\]
Since Definition \ref{defn:macdonald} is symmetric if we interchange $q$ and $t$ simultaneously replacing $\lambda$ by $\lambda'$ we obtain
\[
\tilde H_\lambda[(q-1) X;q,t] = c_\lambda(q,t) m_{\lambda'} + \langle m_\nu: \nu\preceq\lambda'\rangle,
\]
where $c_\lambda(q,t)=\prod_{a,l}(q^{a+1}-t^l)$.
Substitution $t=0$ gives
\[
H_\lambda[(q-1) X;q] = c_\lambda(q,0) m_{\lambda'} + \langle m_\nu: \nu\preceq\lambda'\rangle.
\]
So \eqref{eq:C to H} implies
\[
C_{\lambda,q}[(q-1) X;q] = c_\lambda(q,0) m_{\lambda'} + \langle m_\nu: \nu\preceq\lambda'\rangle.
\]
Thus we have
\[
C_{\lambda,q}[1;t] = \frac{c(q,0)}{c(q,t)} = \prod_{\substack{a,l\\l\neq 0}}\frac{1}{1-t^l q^{-a-1}}.
\]
\end{proof}

Thus the complete formula for the weighted count reads
\begin{cor}\label{cor:count gets macdonald}
\[
C_{\lambda, q}[X;t] = \sum_{[\eta]\in\ANilp_\lambda(\BF_q)} t^{\deg\eta} \weight(\eta) H_{\type \eta(0)}[X;q] = \frac{\tilde H_\lambda[X;q,t]}{\prod_{a,l:l\neq 0}(1-t^l q^{-a-1})}.
\]
\end{cor}

As an interesting experiment one may try to make sense of Corollary \ref{cor:genus g} when $S$ is the collection of \emph{all} points of $\Sigma$: let us count pairs $(\CE,\theta)$ such that $\theta$ is nowhere degenerate. Proof of the following is left to the reader:
\begin{cor}
For a smooth complete curve $\Sigma/\BF_q$ and a partition $\lambda$ we have
\[
\sum_{\CE\in\Bun^{\leq 0}(\Sigma)/\sim} \frac{t^{-\deg\CE}}{|\Aut\CE|} |\theta:\CE\to\CE \;\text{nilpotent, $\type \theta(s)=\lambda$ for all $s\in\Sigma$}|
\]
\[
= \Omega_{g,\lambda}(q,t,\sigma) \prod_{a,l:l\neq 0} \frac{1}{\zeta_\Sigma(t^l q^{-a-1})}.
\]
\end{cor}

Similarly, one can obtain counts for nilpotent endomorphisms with arbitrary prescribed types at all points. Of course, for all but finitely many points the type should be equal to the generic type. Otherwise the answer is zero.

\subsection{Nilpotent affine Springer fiber}\label{ssec:affine springer}
Let us interpret Corollary \ref{cor:count gets macdonald} in terms of the affine Springer fiber. We will use $\Bk=\BF_q$, $\BFR=\Bk[[x]]$, $\BFK=\Bk((x))$ and other notations from Section \ref{sec:linear algebra}. Let $\lambda, \mu$ be partitions of $n$. To $\lambda$ we associate the constant nilpotent matrix $N_\lambda$. To $\mu$ we associate the Iwahori subgroup
\[
I_\mu = \{g\in\GL_n(\BFR):g(0)\in P_\mu\}
\]
where $P_\mu\subset\GL_n(\Bk)$ is the corresponding parabolic subgroup, i.e. the group of block-upper-triangular matrices with blocks of sizes $\mu_1,\mu_2,\ldots$. Then the \emph{partial affine flag variety} is defined by
\[
\AFl_\mu = \GL_n(\BFK)/I_\mu.
\]
We have a natural projection $\AFl_\mu\to\AGr_n$. The \emph{affine Springer fiber} over a matrix $M\in\Mat_n(\BFR)$ is defined as the subset of flags respected by $M$, i.e.
\[
\ASpr_\mu(M) = \{[g]\in\AFl_\mu: g^{-1}M g\in \Lie I_\mu\},
\]
where
\[
\Lie I_\mu = \{m\in\Mat_n(\BFR):m(0)\in \Lie P_\mu\}
\]
and $\Lie P_\mu \subset \Mat_n(\Bk)$ is the subset of block-upper-triangular matrices (not necessarily invertible).

We would like to interpret the coefficients of the Macdonald polynomial in the monomial basis as counting points with weights on the affine Springer fibers. The centralizer of $N_\lambda$ in $\GL_n(\BFK)$, denoted by $Z_\BFK(N_\lambda)$, naturally acts on $\ASpr_\mu(N_\lambda)$ on the left and we have a natural identification 
\[
Z_\BFK(N_\lambda)\backslash\ASpr_\mu(N_\lambda)\cong\{[g]\in Z_\BFK(N_\lambda)\backslash \GL_n(\BFK)/I_\mu: g^{-1} N_\lambda g \in\Lie I_\mu\}\cong\ANilp_{\lambda,\mu}(\Bk),
\]
where
\[
\ANilp_{\lambda,\mu}(\Bk)= \{\theta\in\Lie I_\mu: \text{nilpotent, $\type\theta=\lambda$}\}/\{\theta\sim g\theta g^{-1}:g\in I_\mu\}.
\]
On the other hand, we have
\[
\ANilp_\lambda(\Bk)\cong \{\theta\in\Mat_n(\BFR): \text{nilpotent, $\type\theta=\lambda$}\}/\{\theta\sim g\theta g^{-1}:g\in GL_n(\BFR)\}.
\]
So there is a natural map 
\[
\pi:\ANilp_{\lambda,\mu}(\Bk)\to \ANilp_\lambda(\Bk)
\]
whose fiber over $[\theta]$ can be identified with 
\[
Z(\theta)\backslash \Fl_\theta(\Bk), \qquad \Fl_\theta(\Bk)=\{[g]\in \GL_n(\BFR)/I_\mu: g^{-1} \theta g \in\Lie I_\mu\},
\]
where $Z(\theta)$ is the centralizer of $\theta$ in $\GL_n(\BFR)$. We see that the set $\Fl_\theta(\Bk)$ is the set of flags preserved by $\theta(0)$. By orbit counting we obtain
\[
(C_{\lambda, q}[X;t], h_\mu[X])= \sum_{[\theta]\in\ANilp_\lambda(\Bk)} t^{\deg\theta} \weight(\theta) |\Fl_\theta(\Bk)| 
\]
\[
= \sum_{[\theta]\in\ANilp_{\lambda,\mu}(\Bk)} t^{\deg\theta} \weight(\theta)\; |Z(\theta)/(Z(\theta)\cap I_\mu)|.
\]
Let us denote for any $\theta\in\Lie I_\mu$
\[
\weight_\mu(\theta) = \weight(\theta)\; |Z(\theta)/(Z(\theta)\cap I_\mu)|.
\]
\begin{rem}
Recall the notion of classification data $M_\theta$ from Section \ref{ssec:classification data}.  Note that because the action of $Z(\theta)$ on $M_\theta$ is free, the new weight can be defined similarly to Definition \ref{defn:weight} so that we have
\[
\weight_\mu(\theta) = \frac{|M_\theta/(Z(\theta)\cap I_\mu)|}{|Z(N_\lambda)\backslash M_\theta|}.
\]
Similarly to Remark \ref{rem:weight commensurability} we also have that the weight can be interpreted as the commensurability index:
\[
\weight_\mu(\theta) = [Z(N_\lambda):g (Z(\theta)\cap I_\mu) g^{-1}].
\]
for any $g\in M_\theta$.
\end{rem}
The conclusion is
\begin{thm}\label{thm:springer fiber}
For any $n$ and partitions $\lambda,\mu\vdash n$ and a finite field $\Bk$ of size $q$ we have a bijection $Z_\BFK(N_\lambda)\backslash\ASpr_\mu(N_\lambda)\cong\ANilp_{\lambda,\mu}(\Bk)$
and 
\[
\sum_{[\theta]\in\ANilp_{\lambda,\mu}(\Bk)} t^{\deg\theta} \weight_\mu(\theta) = \frac{(\tilde H_\lambda[X;q,t],h_\mu[X])}{\prod_{a,l:l\neq 0}(1-t^l q^{-a-1})},
\]
where the product on the right hand side goes over the arm- and leg-lengths $a,l$ of the hooks of $\lambda$ such that $l\neq 0$.
\end{thm}

\section{The category of parabolic sheaves}\label{sec:category parabolic}
In this section we study the category of parabolic coherent sheaves, see \cite{heinloth2004coherent}. Unfortunately we could not find sufficiently general results with proofs about this category in the literature, so we build the theory from scratch below.

\subsection{Parabolic sheaves}\label{ssec:parabolic sheaves}
Let $\Sigma$ be a smooth complete curve over a field $\Bk$ and let $S=(s_1,s_2,\ldots,s_k)$ be a collection of distinct closed points of $\Sigma$ of degrees $d_1,d_2,\ldots,d_k$. Let $N$ be an integer. Let $\BI_N=\{0,\ldots,N-1\}$. Throughout this section we keep $\Sigma$ and $N$ fixed, so sometimes we omit them from the notation.
For a collection of sheaves $\CF_{\bar i}$ indexed by vectors $\bar i=(i_1,\ldots,i_k)\in \BI_N^k$ we define $\CF_{\bar i}$ for any vector $\bar i\in\BZ^k$ by
\[
\CF_{\bar i} = \CF_{i_1\% N, \ldots, i_k\% N}\otimes \CO\left(\sum_{j=1}^k \left\lfloor \frac{i_j}{N} \right\rfloor [s_j]\right),
\]
where $a\% b$ denotes the residue of $a$ modulo $b$. Denote by $\delta_j$ the vector with coordinates $\delta_{j,j'}$.
\begin{defn}
A \emph{parabolic quasi-coherent sheaf} is a collection of quasi-coherent sheaves $\CF=(\CF_{\bar i}:\bar i\in I^k)$ together with morphisms
\[
\varphi_{j}:\CF_{\bar i}\to\CF_{\bar i + \delta_j}\qquad(\bar i \in \BI_N^k, j\in \BI_N),
\]
satisfying the following properties:
\begin{enumerate}
\item For any $j, j'$ we have $\varphi_j \varphi_{j'} = \varphi_{j'} \varphi_j$,
\item for any $j$ the morphism $\varphi_j$ restricted to $\Sigma\setminus\{s_j\}$ is an isomorphism,
\item for any $j$ the $N$-th iteration $\varphi_j^N:\CF_{\bar i} \to \CF_{\bar i+N\delta_j} = \CF_{\bar i}(s_j)$ coincides with the natural map.
\end{enumerate}
A \emph{parabolic coherent sheaf} is a parabolic quasi-coherent sheaf whose components are coherent. The category $\Par_{N,S}(\Sigma)$ of parabolic coherent sheaves is defined as the category whose objects are parabolic coherent sheaves and whose morphisms $\CF\to\CF'$ are collections of morphisms $\CF_{\bar i}\to \CF_{\bar i}'$ commuting with all $\varphi_j$. Similarly we define the category of parabolic quasi-coherent sheaves $\QPar_{N,S}(\Sigma)$.
\end{defn}

The category of parabolic coherent resp. quasi-coherent sheaves is abelian with component-wise kernels and cokernels. If any $\CF_{\bar i}$ is a vector bundle then all $\CF_{\bar i}$ are vector bundles and all $\varphi_j$ are injective. Such parabolic sheaves are called \emph{parabolic bundles}. The data of a parabolic bundle is simply the data of a bundle $\CF_{0,\ldots,0}=\CF$ and a filtration of length $N$ of each fiber $\CF(s_i)$ by vector spaces. Note that the condition ``$\CF$ is a parabolic bundle and $\CF_0\in\Bun^{\leq 0}$'' is a suitable truncation in the sense of Definition \ref{defn:suitable truncation}. We denote by $\ParBun_{N,S}(\Sigma)$ the additive category of parabolic bundles, and by $\ParBun_{N,S}^{\leq 0}(\Sigma)$ the additive category of parabolic bundles satisfying $\CF_0\in\Bun^{\leq 0}$.

The \emph{rank} and the \emph{degree} of a parabolic coherent sheaf $\CF\in\Par_S$ is defined as the rank and the degree of the corresponding $(0,\ldots,0)$-component. The numerical invariants $r_{i,j}$ called \emph{parabolic jumps} for $i=1,\ldots,k$ and $j=1,\ldots,N$ are defined by
\[
r_{i,j}(\CF) = \frac{\deg \CF_{j\delta_i} - \deg \CF_{(j-1)\delta_i}}{d_i}.
\]
Note that for all $i$ we have
\[
\sum_{j=1}^N r_{i,j}(\CF)= \rank\CF.
\]
The rank, degree and $r_{i,j}$ are all additive for short exact sequences.

Any usual sheaf $\CF$ can be viewed as a parabolic sheaf by setting $\CF_{\bar i}=\CF$ for all $\bar i\in\BI_N^k$ and all the maps $\varphi_j$ the identity maps except $\varphi_j:\CF=\CF_{\bar i}\to\CF_{\bar i+\delta_j}=\CF(s_j)$ the natural map when $i_j=N-1$.

\subsection{Iterated construction}
Alternatively, one can define the category $\Par_{S}$ by adding one point at a time. We set $\Par_{\varnothing}=\Coh(\Sigma)$. Then for any $S$ let $S'=(s_1,\ldots,s_{k-1})$, $s=s_k$. The category $\Par_{S}$ is defined as the category of \emph{necklaces} of objects of $\Par_{S'}$:
\[
\cdots \to \CF_{N-1}(-s) \to \CF_0 \to\cdots\to \CF_{N-1}\to \CF_0(s) \to \cdots,
\]
where the above sequence is periodic, each map is an isomorphism on $\Sigma\setminus\{s\}$ and the composition of any $N$ consecutive maps is the natural map $\CF_i\to\CF_i(s)$.

Let $R_i:\Par_{S}\to \Par_{S'}$ be the functor that sends a necklace $\CF_\bullet$ to $\CF_i$. Let $I_i:\Par_{S'} \to \Par_{S}$ be the functor that sends $\CE\in \Par_{S'}$ to $\CF_\bullet$ defined by $\CF_{i+j}=\CE$ for $0\leq j< N$ with all the maps the identity, except $\CE\to\CE(s)$ which is the natural map. We have natural identifications for all $\CE\in \Par_{S'}$, $\CF\in\Par_{S}$:
\begin{equation}\label{eq:adjunction}
\Hom(\CE, R_i\CF) = \Hom(I_i \CE, \CF),\qquad \Hom(R_i\CF, \CE) = \Hom(\CF, I_{i-N+1} \CE)
\end{equation}
which can be seen from the following diagrams
\[
\begin{tikzcd}
\cdots \arrow{r} & \CE(-s) \arrow{r}\arrow{d}  & \CE \arrow{r}{=}\arrow{d} & \CE \arrow{r}{=}\arrow{d}  & \CE \arrow{r}{=}\arrow{d}  & \cdots\\
\cdots \arrow{r}& \CF_{i-1}\arrow{r} & \CF_{i}\arrow{r} & \CF_{i+1}\arrow{r} & \CF_{i+2}\arrow{r} & \cdots
\end{tikzcd}
\]
\[
\begin{tikzcd}
\cdots \arrow{r}& \CF_{i-2} \arrow{r}\arrow{d} & \CF_{i-1} \arrow{r}\arrow{d} & \CF_{i} \arrow{r}\arrow{d} & \CF_{i+1} \arrow{r}\arrow{d} & \cdots\\
\cdots \arrow{r}{=} & \CE \arrow{r}{=} & \CE \arrow{r}{=}  & \CE  \arrow{r}  & \CE(s) \arrow{r} & \cdots
\end{tikzcd}
\]

The functions $R_i$ and $I_i$ are defined similarly for quasi-coherent sheaves. Note that both $R_i$ and $I_i$ are exact. Hence category $\QPar_S$ has enough injectives. This can be seen by induction because $I_i$ preserves injectives and one can embed any object $\CE\in\QPar_S$ into the direct sum 
\[
\CE\to\bigoplus_{i=0}^{N-1} I_{i-N+1}(\CE_i).
\]
So we define the higher $\Ext$ functors on $\QPar_S$ using injective resolutions and then restrict the definition to $\Par_S$.
Exactness of $R_i$ and $I_i$ implies
\begin{prop}
The natural adjunctions extend to the Ext functors:
\[
\Ext^j(\CE, R_i\CF) \cong \Ext^j(I_i \CE, \CF),\qquad \Ext^j(R_i\CF, \CE) \cong \Ext^j(\CF, I_{i-N+1} \CE)
\]
for all $i$ and $j$.
\end{prop}

\subsection{Generators and Euler form}
For any $i$ denote by $\Bk_{s,i}\in\Par_{S}$ the object whose $i$-th component is the skyscraper sheaf at $s$ and all the other components are $0$.

We have 
\begin{prop}\label{prop:generating}
The category $\Par_{S}$ is generated under extensions by objects of the form $I_0(\CE)$ for $\CE\in\Par_{S'}$ and objects of the form $\Bk_{s,i}$ for $i=0,\ldots,N-1$.
\end{prop}
\begin{proof}
For any $\CF\in\Par_{S}$ consider the adjunction map $\iota:I_0R_0 \CF \to \CF$ and take its cokernel $\CG=\cokernel \iota$.
We have that $\CG$ satisfies $R_0 \CG=0$. Such sheaves are simply representations of the $A_{N-1}$ quiver
\[
0\to \CG_1\to\cdots\to \CG_{N-1}\to 0
\]
and each $\CG_i$ is a direct sum of finitely many skyscrapers at $s$. So $\CG$ can be obtained by extensions from objects of the form $\Bk_{s,i}$. Let $\CE=\image\iota$. We have that each map $\CE_0\to\CE_{i}$ for $0\leq i\leq N-1$ is surjective. Now consider another adjunction map $\iota': \CE\to I_0 R_{N-1}\CE$. This map is surjective and its kernel $\CG'$ satisfies $R_{N-1} \CG'=0$. Thus again it corresponds to a representation of the $A_{N-1}$ quiver and can be obtained by extensions from $\Bk_{s,i}$. Then we have that $\CE$ is an extension of $I_0 R_{N-1}\CE$ by $\CG'$ and $\CF$ is an extension of $\CG$ by $\CE$.
\end{proof}

\begin{prop}\label{prop:hereditary}
The category $\Par_{S}$ is hereditary. Let $\chi_S$ denote the Euler form on $K(\Par_S)$. We have 
\[
\chi_S(I_0(\CE),\CF) = \chi_{S'}(\CE, \CF_0),\quad \chi_S(\Bk_{s,i}, \CF) = - d_k d_{k,i+1}(\CF)
\]
for all $\CF\in\Par_S$, $\CE\in\Par_{S'}$, $i=0,1,\ldots,N-1$.
\end{prop}
\begin{proof}
We proceed by induction in the size of $S$. It is enough to show that any object $\CE$ from the generating set of Proposition \ref{prop:generating} satisfies $\Ext^j(\CE, \CF)=0$ for any $\CF\in\Par_{S}$ and $j>1$.

For any $\CE\in\Par_{S'}$ we have $\Ext^j(I_0(\CE), \CF)=\Ext^j(\CE, \CF_0 )$ and vanishes for $j>1$ by the induction hypothesis.

For the objects $\Bk_{s,i}$ we use the short exact sequence
\[
0\to I_{i+1}(\CO) \to I_i(\CO)\to \Bk_{s,i}\to 0
\]
to obtain the long exact sequence for any $\CF\in\Par_{S}$.
\[
0\to\Hom(\Bk_{s,i},\CF) \to \Hom(\CO, \CF_i) \to  \Hom(\CO, \CF_{i+1})
\]
\[
\to \Ext^1(\Bk_{s,i},\CF) \to \Ext^1(\CO, \CF_i) \to \Ext^1(\CO, \CF_{i+1})\to \Ext^2 (\Bk_{s,i},\CF) \to 0.
\]
Since $\CO$ is coming from $\Coh(\Sigma)$ we can replace $\CF_i$, $\CF_{i+1}$ by their $(0,\ldots,0)$-components, which are objects in $\Coh(\Sigma)$.
It is enough to show that the map $\Ext^1(\CO, \CF_i) \to \Ext^1(\CO, \CF_{i+1})$ is surjective. This will follow from surjectivity of the following composition:
\begin{equation}\label{eq:ext map surjective}
\Ext^1(\CO, \CF_{i+1}(-s)) \to \Ext^1(\CO, \CF_i) \to \Ext^1(\CO, \CF_{i+1}),
\end{equation}
where $\CF_{i+1}(-s)\to \CF_{i+1}$ is the natural map. The composition can also be described as
\[
\Ext^1(\CO(s), \CF_{i+1}) \to \Ext^1(\CO, \CF_{i+1}),
\]
which is surjective because it is the end of the long exact sequence for $\Hom(-,\CF_{i+1})$ applied to 
\[
0\to \CO \to \CO(s) \to \Bk_s \to 0.
\]
Thus  \eqref{eq:ext map surjective} is surjective and $\Ext^2$ vanishes. The values of the Euler form are determined from the exact sequences above.
\end{proof}

Next we determine the Euler form:
\begin{prop}\label{prop:euler form}
The Euler form on $\Par_S$ is given by
\[
\chi_{\Par_S}(\CE, \CF) = (1-g) \rank \CE \rank \CF + \rank\CE \deg \CF -\rank\CF \deg \CE
\]
\nopagebreak
\[
- \sum_{i=1}^k d_i \sum_{1\leq j<j'\leq N} r_{i,j}(\CE) r_{i,j'}(\CF).
\]
\end{prop}
\begin{proof}
For any $\CE\in\Par_{S'}$ we have
\[
\rank I_0(\CE)=\rank\CE,\; \deg I_0(\CE)=\deg \CE,\; r_{i,j} I_0(\CE) = 
\begin{cases}r_{i,j}(\CE) & i<k\\
0 & i=k,j<N\\
\rank(\CE) & i=k,j=N.
\end{cases}
\]
For $\Bk_{s,j}$ with $j=1,\ldots,N-1$ we have 
\[
\rank \Bk_{s,j}=0,\; \deg \Bk_{s,j}=0,\; r_{i,j'}(\Bk_{s,j}) = 
\begin{cases}0 & \text{$i<k$ or $i=k$, $j'\notin\{j,j+1\}$}\\
1 & i=k,j'=j\\
-1 & i=k,j'=j+1.
\end{cases}
\]
For $\Bk_{s,0}$ we have 
\[
\rank \Bk_{s,0}=0,\; \deg \Bk_{s,0}=d,\; r_{i,j'}(\Bk_{s,0}) = 
\begin{cases}0 & \text{$i<k$ or $i=k$, $j'\notin\{1,N\}$}\\
1 & i=k,j'=N\\
-1 & i=k,j'=1.
\end{cases}
\]
By Proposition \ref{prop:generating} it is enough to check the formula for objects of the above three types in place of $\CE$ and arbitrary objects in place of $\CF$. Inserting the values into the formula and comparing with Proposition \ref{prop:hereditary} we see that the formula is correct.
\end{proof}

\subsection{Serre duality}
Finally we have
\begin{thm}[Serre duality]\label{thm:Serre duality}
Let $D:\Par_S \to \Par_S$ be an auto-equivalence defined by 
\[
D(\CF)_{\bar i} = \Omega^1_\Sigma\otimes \CF_{(N-1)\sum_{i}\delta_i+\bar i}.
\]
We have the following natural isomorphism for all $\CE, \CF\in\Par_S$:
\[
\Ext^1(\CE, \CF)^* \cong \Hom(\CF, D \CE),
\]
where $*$ denotes the dualization over the base field.
\end{thm}
\begin{proof}
As usual, we proceed by induction on the number of marked points $k$. The base case $k=0$ is the usual Serre duality. Now suppose we have the statement for $S'=S\setminus\{s\}$, $s=s_k$. Denote the functor $D$ on the category $\Par_{S}$ by $D_S$. We have for any $\CF_\bullet\in\Par_S$ viewed as a necklace
\[
D_S(\CF)_i = D_{S'}(\CF_{i+N-1}).
\]
For any $\CE,\CF\in\Par_S$ define 
\[
\widetilde\CE = \bigoplus_{i=0}^{N-1} I_{i-N+1}(\CE_i),\quad \widetilde\CF = \bigoplus_{i=0}^{N-1} I_{i}(\CF_i).
\]
The adjunction maps give us an injection $\CE\to\widetilde\CE$ and a surjection $\widetilde\CF\to\CF$, so we obtain short exact sequences
\[
0\to\CE\to\widetilde\CE\to\CE'\to 0,\quad 0\to\CF'\to\widetilde\CF\to\CF\to 0.
\]
Denote by $H(\CG,\CG')$ be any of the two functors $\Ext(\CG,\CG')^*$, $\Hom(\CG',D_S \CG)$. Note that $D_S$ is exact. Thus we have an exact square
\[
\begin{tikzcd}
0\arrow[r] & H(\CE, \CF')\arrow[r] & H(\widetilde \CE, \CF')\arrow[r] & H(\CE', \CF')\\
0\arrow[r] & H(\CE, \widetilde\CF) \arrow[r]\arrow[u] & H(\widetilde \CE, \widetilde\CF)\arrow[r] \arrow[u]& H(\CE', \widetilde\CF)\arrow[u]\\
0\arrow[r] & H(\CE, \CF) \arrow[r]\arrow[u] & H(\widetilde \CE, \CF)\arrow[r]\arrow[u] & H(\CE', \CF)\arrow[u]\\
 & 0\arrow[u] & 0\arrow[u] & 0\arrow[u]
\end{tikzcd}
\]
and by diagram chasing a short exact sequence
\begin{equation}\label{eq:H short sequence}
0\to H(\CE, \CF) \to H(\widetilde \CE, \widetilde\CF) \to H(\widetilde \CE, \CF')\oplus H(\CE', \widetilde\CF).
\end{equation}
For any $\CG'\in\Par_{S'}$, any $i\in\BZ$ and any $\CG\in\Par_S$ we have the following sequence of natural isomorphisms:
\[
\Ext^1(I_i(\CG'), \CG)^* \cong \Ext^1(\CG', \CG_i)^* \cong\Hom(\CG_i, D_{S'}\CG') \cong \Hom(\CG, I_{i-N+1}(D_{S'} \CG'))
\]
\[
 = \Hom(\CG, D_S I_i(\CG'))
\]
and similarly
\[
\Ext^1(\CG, I_i(\CG'))^* \cong \Hom(\CG', D_{S'} \CG_{i+N-1})\cong\Hom(I_i(\CG'), D_S \CG').
\]
Applying this to each direct summand of $\widetilde \CE$ resp. $\widetilde \CF$ we obtain for all $\CG\in\Par_S$
\begin{equation}\label{eq:two constructions}
\Ext^1(\widetilde \CE, \CG)^* \cong \Hom(\CG, D_S \widetilde\CE)\quad\text{resp.}\quad
\Ext^1(\CG, \widetilde \CE)^* \cong \Hom(\CE, D_S \widetilde\CG).
\end{equation}
Since these isomorphisms are natural in $\CG$ we can fit the two sequences \eqref{eq:H short sequence} into a commutative diagram
\[
\begin{tikzcd}
0\arrow[r]& \Ext^1(\CE, \CF)^* \arrow[r]& \Ext^1(\widetilde \CE, \widetilde\CF)^* \arrow[r] \arrow{d}{\cong}& \Ext^1(\widetilde \CE, \CF')^*\oplus \Ext^1(\CE', \widetilde\CF)^* \arrow{d}{\cong}\\
0\arrow[r]& \Hom(\CF, D_S \CE) \arrow[r]& \Hom(\widetilde \CF, D_S \widetilde\CE) \arrow[r]& \Hom(\CF', D_S \widetilde\CE)\oplus \Hom(\widetilde\CF, D_S \CE')
\end{tikzcd},
\]
provided that we show that the two constructions of the isomorphism $\Ext^1(\widetilde \CE, \widetilde\CF)^*\cong \Hom(\widetilde \CF, D_S(\widetilde\CE))$ obtained by use of the two isomorphisms \eqref{eq:two constructions} respectively agree. Indeed, then the desired isomorphism $\Ext^1(\CE, \CF)^*\cong \Hom(\CF, D_S \CE)$ is uniquely obtained from the diagram.

Since $\widetilde\CE$ and $\widetilde\CG$ are direct sums of corresponding objects it is enough to compare for any $i,j$ and $\CG$, $\CG'$ the following
\[
\Ext^1(I_i(\CG'), I_j(\CG))^* \cong \Ext^1(\CG', R_i(I_j(\CG)))^*  \cong \Hom(R_i(I_j(\CG)), D_{S'}\CG')
\]
\[
 \cong \Hom(I_j(\CG), I_{i-N+1}(D_{S'}\CG')),
\]
\[
\Ext^1(I_i(\CG'), I_j(\CG))^* \cong \Ext^1(R_{j+N-1} I_i(\CG'), \CG)^* \cong \Hom(\CG, D_{S'} R_{j+N-1} I_i(\CG'))
\]
\[
\cong \Hom(I_j(\CG), I_{i-N+1}(D_{S'}\CG')). 
\]
Let $m=\lfloor\frac{i-j}{N}\rfloor$. We have $R_i(I_j(\CG)) = \CG(m s)$, $\lfloor\frac{j+N-1-i}{N}\rfloor=-m$, $R_{j+N-1} I_i(\CG')=\CG'(-m)$, so the fact that the two isomorphisms agree follows from commutativity of the diagram
\[
\begin{tikzcd}
\Ext^1(\CG', \CG(ms))^* \arrow{r}{\cong}\arrow{d}{\cong} &\Hom(\CG(ms), D_{S'} \CG')\arrow{d}{\cong} \\
\Ext^1(\CG'(-ms), \CG)^* \arrow{r}{\cong} &\Hom(\CG, D_{S'} \CG'(-ms))
\end{tikzcd}
\]
The fact that the diagram commutes can be shown by induction starting from the corresponding diagram in $\Coh(\Sigma)$ and noticing that all the steps in our construction of Serre duality behave as expected with respect to twists by line bundles.
\end{proof}

Note that for any $\CE\in\Par_S$ we have
\[
\rank D\CE = \rank\CE,\quad \deg D\CE = \deg\CE+\rank\CE\left(2g-2+\sum_{i=1}^k d_i\right) - \sum_{i=1}^k d_i r_{i,N},
\]
\[
r_{i,j}(D\CE) = r_{i,j-1} (\CE).
\]
\subsection{Parabolic Higgs bundles}
Strictly speaking, we cannot apply the well-known results of Gothen and King \cite{gothen2005homological} in our situation because parabolic sheaves are not sheaves. However, it is quite easy to generalize their results. Let $\CA$ be an abelian category and $D:\CA\to\CA$ any exact endofunctor. Consider the category $\CA_D$ of pairs $\overline\CE=(\CE,\theta)$ where $\CE\in\CA$, $\theta\in\Hom(\CE,D\CE)$. This is an abelian category with obvious definitions of morphisms, kernels and cokernels. We then have
\begin{thm}\label{thm:higgs category}
Suppose $\CA$ is closed under countable products and has enough injectives.  Then $\CA_D$ has enough injectives and we have the following functorial long exact sequence for any $\overline\CE=(\CE,\theta)$, $\overline\CE'=(\CE,\theta')$ in $\CA_D$:
\[
0\to \Hom(\overline\CE, \overline\CE') \to \Hom(\CE,\CE') \xrightarrow{\theta'\circ - D(-)\circ\theta} \Hom(\CE,D \CE')
\]
\[
\to \Ext^1(\overline\CE, \overline\CE') \to \Ext^1(\CE,\CE') \xrightarrow{\theta'\circ - D(-)\circ\theta} \Ext^1(\CE,D \CE')\to\cdots
\]
\end{thm}
\begin{proof}
The forgetful functor $L:\CA_D\to\CA$ has an exact right adjoint $R$ constructed as follows. For any $\CF\in\CA$ take
\[
R(\CF) = \left(\prod_{i=0}^\infty D^i \CF, \text{shift}\right),
\]
where 
\[
\text{shift}: \prod_{i=0}^\infty D^i \CF \to D\left(\prod_{i=0}^\infty D^i \CF\right)=\prod_{i=0}^\infty D^{i+1} \CF
\]
 is defined as the identity on each component $D^i\CF\xrightarrow{=} D(D^{i-1}\CF)$ with an obvious shift of indices. Then any map $f:\CE\to\CF$ for $\overline\CE=(\CE,\theta)\in\CA_D$, $\CF\in\CA$ uniquely extends to a morphism $(\CE,\theta)\to(R\CF,\text{shift})$. Its $i$-th component is given by $D^i(f) \theta^i$. So we have adjunction
 \[
\Hom(L \overline\CE, \CF) \cong \Hom(\overline\CE, R \CF).  
 \]
Since $L$ is exact, $R$ preserves injectives. For any $\overline\CE\in\CA_D$ the natural map $\overline\CE \to RL\overline\CE$ is an injection, hence $\CA_D$ has enough injectives. Moreover, the adjunction extends to all the higher $\Ext$ functors
 \begin{equation}\label{eq:adjunctions higgs}
\Ext^i(L \overline\CE, \CF) \cong \Ext^i(\overline\CE, R \CF).
 \end{equation}
For any $\overline\CE=(\CE,\theta)\in\CA_D$ we have a short exact sequence
\[
0\to \overline\CE \to R \CE \to RD\CE \to 0,
\]
where the $i$-th component of the map $R\CE \to RD\CE$ is given by the difference 
\[
\Id_{D^i\CE} - D^i(\theta): D^i \CE \to D^{i} \CE \oplus D^{i+1}\CE.
\]
For any $\overline{F}\in\CA_D$ we obtain the long exact sequence
\[
0\to \Hom(\overline{F}, \overline\CE) \to \Hom(\overline{F}, R \CE) \to \Hom(\overline{F}, RD\CE)\to\cdots.
\]
Applying adjunctions \eqref{eq:adjunctions higgs} we obtain the desired long exact sequence.
\end{proof}

In the situation $\CA=\QPar_{N,S}(\Sigma)$ and $D$ the Serre functor of Theorem \ref{thm:Serre duality} the category $\CA_D$ is called the category of \emph{parabolic quasi-coherent Higgs sheaves}. The category $\ParHiggs_{N,S}(\Sigma)$ of \emph{parabolic Higgs bundles} resp. \emph{parabolic coherent Higgs sheaves} is the full subcategory of pairs $(\CE,\theta)$ such that $\CE$ is a parabolic bundle resp. parabolic coherent sheaf. We obtain the following
\begin{cor}\label{cor:category higgs}
The category of parabolic coherent Higgs sheaves has global dimension $2$. For $\overline\CE,\overline\CF$ parabolic coherent Higgs sheaves the Euler form is given by 
\[
\chi(\overline\CE,\overline\CF)=\chi(\CE,\CF)+\chi(\CF,\CE)=\left(2-2g-\sum_{i=1}^k d_i\right)\rank\CE\rank\CF + \sum_{i=1}^k d_i \sum_{j=1}^N r_{i,j}(\CE) r_{i,j}(\CF)
\]
and we have $\Hom(\overline\CF,\overline\CE)\cong \Ext^2(\overline\CE,\overline\CF)^*$.
\end{cor}
\begin{proof}
By Theorem \ref{thm:higgs category}, since $\QPar_S$ is hereditary we have a long exact sequence
\[
0\to \Hom(\overline\CE, \overline\CF) \to \Hom(\CE,\CF) \to \Hom(\CE,D \CF)
\]
\[
\to \Ext^1(\overline\CE, \overline\CF) \to \Ext^1(\CE,\CF) \to \Ext^1(\CE,D \CF)\to\Ext^2(\overline\CE, \overline\CF)\to 0.
\]
This gives the Euler form
\[
\chi(\overline\CE, \overline\CF) = \chi(\CE,\CF) - \chi(\CE,D\CF) = \chi(\CE,\CF) + \chi(\CF,\CE),
\]
where the last equation holds by Serre duality. Using Proposition \ref{prop:euler form} we obtain the formula for the Euler form. Finally notice that $\Hom(\overline\CF,\overline\CE)$ is 
\[
\kernel(\Hom(\CF,\CE)\to\Hom(\CF,D\CE)) \cong
\kernel(\Ext^1(\CE,D\CF)^*\to\Ext^1(\CE, \CF)^*)
\]
\[
\cong \cokernel(\Ext^1(\CE, \CF)\to \Ext^1(\CE, D\CF))^*\cong \Ext^2(\overline\CE,\overline\CF)^*.
\]
\end{proof}

\subsection{Harder-Narasimhan theory}
We apply Harder-Narasimhan theory (see \cite{harder1974cohomology} or \cite{bridgeland2007stability}) to the categories $\ParBun_{N,S}(\Sigma)$ and $\ParHiggs_{N,S}(\Sigma)$.
\begin{defn}
A \emph{stability condition} on $\ParBun_{N,S}(\Sigma)$ or $\ParHiggs_{N,S}(\Sigma)$ is a collection of numbers $\alpha=(\alpha_{i,j})$, $\alpha_{i,j}\in\BR$ for $i=1,\ldots,k$ and $j=1,\ldots,N$ satisfying the following condition:
\[
\alpha_{i,1}\geq \alpha_{i,2}\geq\cdots\cdots\alpha_{i,N}\geq \alpha_{i,1}-d_i.
\]
Note that $\alpha_{i,j}=0$ is a valid stability condition, which we denote by $0$.
The $\alpha$-degree of $\CE\in\Par_{N,S}(\Sigma)$ is defined by 
\[
\deg_\alpha \CE = \deg\CE+\sum_{i=1}^k\sum_{j=1}^N \alpha_{i,j}r_{i,j}(\CE),
\]
Note that $\deg_\alpha\CE\geq 0$ for all torsion $\CE$.
The \emph{slope} of a parabolic bundle $\CE$ is defined by
\[
\mu_\alpha(\CE) = \frac{\deg_\alpha\CE}{\rank\CE}.
\]
An object $\CF\in\ParBun_{N,S}(\Sigma)$ resp. $\ol\CF=(\CF,\theta)\in\ParHiggs_{N,S}(\Sigma)$ is \emph{semistable} with respect to $\alpha$ if it is not zero and for all non-zero subobjects $\CE\subset\CF$ resp. $\ol\CE\subset\ol\CF$  we have $\mu_\alpha(\CE)\leq\mu_\alpha(\CF)$. If the inequality is strict for all proper subobjects, the object is called \emph{stable}.
\end{defn}
\begin{thm}[Harder-Narasimhan]\label{thm:harder narasimhan}
For any stability condition $\alpha$ and any $\CE\in \ParBun_{N,S}(\Sigma)$ there exists a unique filtration (called Harder-Narasimhan filtration)
\[
0=\CE_0\subset\CE_1 \subset\cdots\subset\CE_{m-1}\subset\CE_m=\CE
\]
such that $\CE_i/\CE_{i-1}$ is a semistable parabolic bundle for $i=1,2,\ldots,m$ and the following holds:
\[
\mu_\alpha(\CE_1) > \mu_\alpha(\CE_2/\CE_1) > \ldots >\mu_\alpha(\CE_m/\CE_{m-1}).
\]
Analogous statement holds for arbitrary $\ol\CE\in\ParHiggs_{N,S}(\Sigma)$.
\end{thm}

Let us denote $\mu_\alpha^{\max}(\CE)$ resp. $\mu_\alpha^{\min}(\CE)$ the numbers $\mu_\alpha(\CE_1)$ resp. $\mu_\alpha(\CE_m/\CE_{m-1})$. They have an alternative description as 
\[
\mu_\alpha^{\wmax}(\CE) = \max_{0\neq \CF\subset\CE} \mu_\alpha(\CF),\quad
\mu_\alpha^{\wmin}(\CE) = \min_{\CE\twoheadrightarrow\CF\neq 0} \mu_\alpha(\CF).
\]
Note that for any two stability conditions $\alpha, \alpha'$ there exist a constant $C\in\BR$ such that for all objects $\CE$
\[
\mu_\alpha^\wmax(\CE)\leq\mu_{\alpha'}^\wmax(\CE)+C,\quad \mu_\alpha^\wmin(\CE)\geq\mu_{\alpha'}^\wmin(\CE)-C.
\]
Note also the following estimate for any parabolic bundle $\CE$:
\begin{equation}\label{eq:mu of D}
\mu_\alpha(D\CE) \leq \mu_\alpha(\CE) + 2g-2+\sum_{i=1}^k d_i.
\end{equation}

Following Mozgovoy and Schiffmann \cite{mozgovoy2014counting} we show
\begin{prop}\label{prop:estimate slope}
For every stability condition $\alpha$ suppose an object $\CE\in\ParBun_{N,S}(\Sigma)$ satisfies one of the following two conditions:
\begin{enumerate}
\item $\CE$ is indecomposable,
\item There exists $\theta:\CE\to D\CE$ such that $(\CE,\theta)$ is semistable.
\end{enumerate}
then we have
\[
\mu^\wmax(\CE) \leq \mu(\CE) + C (\rank \CE-1),\quad \mu^\wmin(\CE) \geq \mu(\CE) - C (\rank \CE-1)
\]
where $C=2g-2+\sum_{i=1}^k d_i$.
\end{prop}
\begin{proof}
Since we have $\mu^\wmax(\CE)\leq\mu(\CE)\leq\mu^\wmin(\CE)$ it is enough to show that $\mu^\wmax(\CE)-\mu^\wmin(\CE)\leq C (\rank\CE-1)$. Note that the Harder-Narasimhan filtration for $\CE$ has at most $\rank\CE$ steps and we have
\[
\mu^\wmax(\CE)-\mu^\wmin(\CE) = \mu(\CE_1) - \mu(\CE_m/\CE_{m-1}) = \sum_{i=1}^{m-1} \mu(\CE_{i}/\CE_{i-1}) - \mu(\CE_{i+1}/\CE_{i}).
\]
Thus it is enough to show that each gap $\mu(\CE_{i}/\CE_{i-1}) - \mu(\CE_{i+1}/\CE_{i})\leq C$. We have a short exact sequence
\[
0\to \CE_i \to \CE \to \CE/\CE_i \to 0
\]
and we have 
\[
\mu(\CE_{i}/\CE_{i-1}) - \mu(\CE_{i+1}/\CE_{i}) = \mu^\wmin(\CE_i) - \mu^\wmax(\CE/\CE_i).
\]

Case 1. Suppose $\CE$ is indecomposable. This implies that the short exact sequence does not split and $\Ext(\CE/\CE_i,\CE_i)\neq 0$. By Serre duality $\Hom(\CE_i, D(\CE/\CE_i))\neq 0$. Thus 
\[
\mu^\wmin(\CE_i)\leq \mu^\wmax(D (\CE/\CE_i))\leq\mu^\wmax(\CE/\CE_i) + C
\]
by \eqref{eq:mu of D}.

Case 2. Suppose $(\CE,\theta)$ is semistable. Then we have that $\theta$ induces a non-zero map $\CE_i\to D(\CE/\CE_i)$, for otherwise $(\CE,\theta)$ would contain a subobject $(\CE_i, D \CE_i)$ whose slope is larger than that of $\CE$. Again, we obtain $\Hom(\CE_i, D(\CE/\CE_i))\neq 0$ and proceed as in Case 1.
\end{proof}

\section{From bundles with nilpotent endomorphism to Higgs bundles and character varieties}\label{sec:from to}
In this section we generalize methods of \cite{mozgovoy2014counting} to the parabolic situation. This generalization is straightforward once we replace the category of coherent sheaves by the category of parabolic coherent sheaves. So we skip through some of the details.

\subsection{Consequences of polynomiality}
Let $\Sigma/\BF_q$ be a smooth complete curve and let $S=(s_1,\ldots,s_k)$ be a collection of points on $\Sigma(\BF_q)$. Although some of the statements will go through for points of arbitrary degree, we restrict our attention to points of to degree $1$ only. Write the zeta function of $\Sigma$ as
\[
\zeta_\Sigma(T) = \frac{\prod_{i=1}^{g}(1-\sigma_i T)(1-q\sigma_i^{-1}T)}{(1-T)(1-qT)},
\]
and denote by $\sigma$ the collection $\sigma=(\sigma_1,\ldots,\sigma_g)$.
Using Corollary \ref{cor:genus 0} in the case if $\BP^1$ or Corollary \ref{cor:genus g} for arbitrary genus we obtain a function
\[
\Omega_{g,k}[X_\bullet;T,q,t,\sigma] = \sum_{\lambda\in\CP} \Omega_{g,\lambda}(q,t,\sigma) T^{|\lambda|}\prod_{i=1}^k \tilde H_\lambda[X_i;q,t]
\]
which counts parabolic bundles with nilpotent endomorphism in the following way. Choose $N\in\BZ_{>0}$ and let $X_i=x_{i,1}+\cdots x_{i,N}$. Recall that for any symmetric function $F$ the plethystic substitution $F[X_i]$ coincides with the usual evaluation $F_N(x_{i,1}, x_{i,2},\ldots,x_{i,N})$ (Section \ref{subs:symmetric functions}). For any $\CE\in \ParBun_{N,S}$ denote 
\[
w(\CE) = T^{\rank\CE} t^{-\deg\CE} \prod_{i=1}^k\prod_{j=1}^N x_{i,j}^{r_{i,j}(\CE)},
\]
see Section \ref{ssec:parabolic sheaves}. We extend the notions of degree, rank and slope to monomials in $T,t,x_{i,j}$ in the obvious way. From \eqref{eq:omega with flags} we obtain
\[
\Omega_{g,k}[X_\bullet;T,q,t,\sigma] = \sum_{\substack{(\CE,\theta)\in\ParBun_{N,S,\nil}^{\leq 0}(\Sigma)/\sim}} \frac{w(\CE)}{|\Aut(\CE,\theta)|},
\]
where $\ParBun_{N,S,\nil}^{\leq 0}(\Sigma)$ is the category of pairs $(\CE,\theta)$ with $\CE\in\ParBun_{N,S}^{\leq 0}(\Sigma)$ and $\theta:\CE\to\CE$ nilpotent. Define $\BH$ by 
\[
\Omega_{g,k}[X_\bullet;T,q,t,\sigma] = \pExp\left[\frac{1}{q-1} \BH_{g,k}[X_\bullet;T,q,t,\sigma]\right].
\]

We define the HLV kernel (see \cite{mellit2016integrality}) of genus $g$ with $k$ punctures by 
\[
\Omega_{g,k}^{\HLV}[X_\bullet;T,q,t,\sigma] = \sum_{\lambda\in\CP} \frac{\prod_{i=1}^{g} N_\lambda(\sigma_i^{-1})}{N_\lambda(1)} T^{|\lambda|}\prod_{i=1}^k \tilde H_\lambda[X_i;q,t],
\]
where 
\[
N_\lambda(u)=\prod_{\square\in\lambda}(q^{a(\square)}-u t^{1+l(\square)})(q^{a(\square)+1}-u^{-1} t^{l(\square)}).
\]
We also define $\BH_{g,k}^{\HLV}$ as a plethystic logarithm so that
\[
\Omega_{g,k}^\HLV[X_\bullet;T,q,t,\sigma] = \pExp\left[\frac{1}{(q-1)(1-t)} \BH_{g,k}^\HLV[X_\bullet;T,q,t,\sigma]\right]
\]
In the case $g=0$ we have
\[
\Omega_{0,k}[X_\bullet;T,q,t] = \Omega_{0,k}^\HLV[X_\bullet;T,q,t],
\]
\[
\BH_{0,k}[X_\bullet;T,q,t] = \frac{1}{1-t}\BH_{0,k}^\HLV[X_\bullet;T,q,t].
\]
The following is shown in \cite{mellit2016integrality}:
\begin{thm}
For all $g,k\in\BZ_{\geq 0}$ the coefficients of $\BH_{g,k}^\HLV[X_\bullet;T,q,t,\sigma]$ in $T$ and $X_\bullet$ are polynomials in $q,t,\sigma_i^{\pm 1}$.
\end{thm}

Then by \cite{mellit2017poincare}, the proof of Theorem 5.2 is easily adopted to show the following\footnote{In \cite{mellit2017poincare} the variables $q,t$ are compared to the variables $q,z$ of Schiffmann as follows: $(q,t)=(z,q)$. Here we use $(q,t)=(q,z)$ instead, so that $q$ stands for the number of elements in the field. At the same time our partitions specifying the nilpotent types are conjugate to those of Schiffmann. We appologize for the inconvenience it may cause.}:
\begin{thm}
For all $g,k\in\BZ_{\geq 0}$ the coefficients of $(1-t)\BH_{g,k}[X_\bullet;T,q,t,\sigma]$ in $T$ and $X_\bullet$ are polynomials in $q,t,\sigma_i^{\pm 1}$ and we have
\[
(1-t)\BH_{g,k}[X_\bullet;T,q,t,\sigma]\Big|_{t=1} = \BH_{g,k}^\HLV[X_\bullet;T,q,t,\sigma]\Big|_{t=1}.
\]
\end{thm}
Thus we obtain
\begin{equation}\label{eq:polynomiality}
\BH_{g,k}[X_\bullet;T,q,t,\sigma] = \frac{1}{1-t} \BH_{g,k}^\HLV[X_\bullet;T,q,t,\sigma] + \text{polynomial part},
\end{equation}
where the coefficients of the polynomial part in $T$ and $X_\bullet$ are polynomials in $t,\sigma_i^{\pm 1}$ over $\BQ(q)$.

\subsection{From bundles with nilpotent endomorphism to bundles with arbitrary endomorphism and indecomposable bundles}
Note that the category $\ParBun_{N,S}^{\leq 0}(\Sigma)$ is a Krull-Schmidt category. In simple terms it means that any endomorphism has a Jordan form decomposition.
Analogously to the proof of Proposition \ref{prop:count nilpotent} we obtain
\[
\pExp[\BH_{g,k}[X_\bullet;T,q,t,\sigma]] = 
\sum_{(\CE,g)\in\ParBun_{N,S,\aut}(\Sigma)/\sim} \frac{w(\CE)}{|\Aut(\CE,g)|}
\]
\[
=\sum_{\CE\in\ParBun_{N,S}^{\leq 0}(\Sigma)/\sim} w(\CE),
\]
where $\ParBun_{N,S,\aut}^{\leq 0}(\Sigma)$ is the category of pairs $(\CE,g)$ with $\CE\in\ParBun_{N,S}^{\leq 0}(\Sigma)$ and $g:\CE\to\CE$ an automorphism. Since any $\CE$ is uniquely a direct sum of indecomposables and any indecomposable is uniquely a base change of a geometrically indecomposable object we obtain
\[
\BH_{g,k}[X_\bullet;T,q,t,\sigma] =
 \sum_{\substack{\CE\in\ParBun_{N,S}^{\leq 0}(\Sigma)/\sim\\ \text{geometrically indecomposable}}} w(\CE).
\]
Twisting by a line bundle of degree $1$ (which always exists) gives an auto-equivalence of $\Par_{N,S}(\Sigma)$ which for any object keeps the rank and the $r_{i,j}$ invariants and increases the degree by the rank. Thus we can always reduce the problem of calculating the number of geometrically indecomposable objects of given rank and degree to the same calculation for a small enough degree. By Proposition \ref{prop:estimate slope} for $\alpha=0$ we conclude
\begin{prop} For any values of $r\in\BZ_{>0}$ and $d\in\BZ$ let $m$ be such that $m\geq (r-1)(2g-2+k)+\frac{d}{r}$. Then we have
\[
\sum_{\substack{\CE\in\ParBun_{N,S}(\Sigma)/\sim\\ \text{geometrically indecomposable}\\ \rank \CE=r,\deg\CE=d}} \prod_{i=1}^k\prod_{j=1}^N x_{i,j}^{r_{i,j}(\CE)} = \BH_{g,k}[X_\bullet;T,q,t,\sigma] \Big|_{T^r t^{d+mr}},
\]
where the notation $\Big|_{T^a t^b}$ means taking the coefficient of the left hand side viewed as a power series in $t$ and $T$ in front of the monomial $T^a t^b$.
\end{prop}
Combining this with \eqref{eq:polynomiality} we have
\begin{cor}\label{cor:indecomposable}
For any values of $r\in\BZ_{>0}$ and $d\in\BZ$ we have
\[
\sum_{\substack{\CE\in\ParBun_{N,S}(\Sigma)/\sim\\ \text{geometrically indecomposable}\\ \rank \CE=r,\deg\CE=d}} \prod_{i=1}^k\prod_{j=1}^N x_{i,j}^{r_{i,j}(\CE)} = \BH_{g,k}^\HLV[X_\bullet;T,q,1,\sigma]\Big|_{T^r}.
\]
In particular, it does not depend on $d$.
\end{cor}

On the other hand, analogously to the proof of Proposition \ref{prop:count nilpotent} but this time for $\BA_1/\BF_q$ instead of $\GL_1/BF_q$ we obtain
\begin{equation}\label{eq:omega with endomorphism}
\pExp\left[\frac{q}{q-1}\BH_{g,k}[X_\bullet;T,q,t,\sigma]\right] = 
\sum_{(\CE,\theta)\in\ParBun_{N,S,\eend}^{\leq 0}(\Sigma)/\sim} \frac{w(\CE)}{|\Aut(\CE,\theta)|},
\end{equation}
where $\ParBun_{N,S,\eend}^{\leq 0}(\Sigma)$ is the category of pairs $(\CE,\theta)$ with $\CE\in\ParBun_{N,S}^{\leq 0}(\Sigma)$ and $\theta:\CE\to\CE$ an endomorphism. The last expression can also be written as
\[
\sum_{\CE\in\ParBun_{N,S}^{\leq 0}(\Sigma)/\sim} \frac{w(\CE)|\Hom(\CE,\CE)|}{|\Aut(\CE)|},
\]

\subsection{From bundles with endomorphism to Higgs bundles}
By Theorem \ref{thm:Serre duality} and Proposition \ref{prop:hereditary} we have for any $\CE\in\ParBun_{N,S}(\Sigma)$
\[
|\Hom(\CE,\CE)| = \frac{|\Hom(\CE,\CE)|}{|\Ext^1(\CE, \CE)|} {|\Ext^1(\CE, \CE)|} = q^{\chi(\CE, \CE)} |\Hom(\CE, D\CE)|.
\]
Let $\ParHiggs_{N,S}(\Sigma)$ be the category of \emph{parabolic Higgs bundles}, i.e. pairs $(\CE,\theta)$ where $\CE$ is a parabolic bundle and $\theta:\CE\to D\CE$ is called a Higgs field. Denote by $\ParHiggs_{N,S}(\Sigma)^{\leq 0}$ the subcategory of pairs $(\CE,\theta)$ such that $\CE\in\ParBun_{N,S}^{\leq 0}(\Sigma)$. From \eqref{eq:omega with endomorphism} we obtain
\begin{equation}\label{eq:omega for higgs}
\pExp\left[\frac{q}{q-1}\BH_{g,k}[X_\bullet;T,q,t,\sigma]\right] = 
\sum_{(\CE,\theta)\in\ParHiggs_{N,S}^{\leq 0}(\Sigma)/\sim} \frac{w(\CE)q^{\chi(\CE,\CE)}}{|\Aut(\CE,\theta)|}.
\end{equation}

\subsection{From all Higgs bundles to semistable Higgs bundles}
Now we consider the Hall algebra $\Hall(\ParBun_{N,S}(\Sigma))$, viewed as a subalgebra of the Hall algebra of parabolic Higgs coherent sheaves. The integration map is given by 
\[
I: \Hall(\ParHiggs_{N,S}(\Sigma)) \to \BQ[[(x_{i,j})_{i=1,j=1}^{k,N},T, t]]\quad I([\overline\CE]) = w(\CE) q^{\chi(\CE,\CE)}.
\]
\begin{prop}\label{prop:integration}
Suppose $\overline\CE,\overline\CF\in\ParHiggs_{N,S}(\Sigma)$ are such that $\Hom(\overline\CF,\overline\CE)=0$. Then we have
\[
I([\overline\CE]*[\overline\CF]) = I([\overline\CE]) I([\overline\CF]).
\]
\end{prop}
\begin{proof}
From the definition of the product the left hand side equals
\[
\frac{|\Ext^1(\ol\CE,\ol\CF)|}{|\Hom(\ol\CE,\ol\CF)|} q^{\chi(\CE,\CE)+\chi(\CE,\CF)+\chi(\CF,\CE)+\chi(\CF,\CF)} w(\CE) w(\CF).
\]
By the assumption from Corollary \ref{cor:category higgs} we obtain $\Ext^2(\CE,\CF)=0$ and therefore
\[
\frac{|\Ext^1(\ol\CE,\ol\CF)|}{|\Hom(\ol\CE,\ol\CF)|} = q^{-\chi(\overline\CE,\overline\CF)} = q^{-\chi(\CF,\CE)-\chi(\CF,\CF)}.
\]
Thus the result follows.
\end{proof}

\begin{prop}\label{prop:omega slope mu}
Let $\alpha$ be a stability condition and let $r\in\BZ_{>0}$. Let $\mu\in\BR$ be such that
\[
\mu\leq -(2g-2+k)(r-1)+\sum_{i=1}^k \alpha_{i,N}.
\]
Consider the sum
\[
\Omega_{\ParHiggs_{N,S}(\Sigma), \alpha,\mu}[X_\bullet;T,t] = 
1+\sum_{\substack{(\CE,\theta)\in\ParHiggs_{N,S}(\Sigma)/\sim \\ \text{semistable, $\mu_\alpha(\CE)=\mu$}}} \frac{w(\CE)q^{\chi(\CE,\CE)}}{|\Aut(\CE,\theta)|}.
\]
On the other hand, decompose $\BH$ as follows:
\[
\BH_{g,k}[X_\bullet;q,t,\sigma] = \sum_{\mu\in\BR} \BH_{g,k,\mu}[X_\bullet;q,t,\sigma],
\]
where $\BH_{g,k,\mu}$ contains only terms of slope $\mu$.
Then we have
\[
\Omega_{\ParHiggs_{N,S}(\Sigma), \alpha,\mu}[X_\bullet;T,t] = 
\pExp\left[\frac{q}{q-1}\BH_{g,k,\mu}[X_\bullet;T,q,t,\sigma]\right] + O(T^{r+1}).
\]
\end{prop}
\begin{proof}
Note that for any stability condition $\alpha$ adding a constant $c\in\BR$ to all $\alpha_{i,j}$ for some $i$ increases the slopes of all objects by $c$. So we can assume that $\alpha_{i,N}= 0$ for all $i$.

We apply Theorem \ref{thm:harder narasimhan} in the following way: for any object $\overline\CE\in\ParHiggs_{N,S}^{\leq 0}(\Sigma)$ of rank $\leq r$ we have a unique filtration 
\[
\ol\CE_0\subset \ol\CE_1\subset\cdots\subset \ol\CE_m=\ol\CE
\]
such that $\mu^\wmin(\ol\CE_0)>\mu$, all $\ol\CE_i$ are semisimple and 
\[
\mu\geq \mu(\CE_1/\CE_0) > \mu(\CE_2/\CE_1) >\cdots.
\]
Since $\CE_0$ is a subobject of $\CE$, we also have $\CE_0\in\ParHiggs_{N,S}^{\leq 0}(\Sigma)$. On the other hand, suppose any object $\ol\CE\in\ParHiggs_{N,S}(\Sigma)^{\leq 0}$ of rank $\leq r$ has a filtration as above such that $\CE_0\in\ParHiggs_{N,S}^{\leq 0}(\Sigma)$. By Proposition \ref{prop:estimate slope} we know that $\mu^\wmax(\CE_i/\CE_{i-1})$ do not exceed $0$ for $i=1,2,\ldots,m$. Because of the condition $\alpha_{i,j}\geq 0$ this implies $\CE_i/\CE_{i-1}\in\ParHiggs_{N,S}^{\leq 0}(\Sigma)$. Thus we obtain $\ol\CE\in\ParHiggs_{N,S}^{\leq 0}(\Sigma)$.

Consider pairings of the form
\[
([\overline\CE], [\overline\CF_m]*\cdots*[\overline\CF_0])=\left(\Delta^{m} [\overline\CE], [\overline\CF_m]\otimes\cdots\otimes[\overline\CF_0]\right)
\]
for $\CF_0\in\ParHiggs_{N,S}^{\leq 0}(\Sigma)$ such that $\mu_\alpha^\wmin(\ol\CF_0)>\mu$, each $\ol\CF_i\in\ParHiggs_{N,S}(\Sigma)$ is semistable and $\mu\geq \mu_\alpha(\overline\CF_1)>\ldots>\mu_\alpha(\overline\CF_m)$. By the above reasoning, we obtain that the pairing is not-zero only if $\ol\CE\in\ParHiggs_{N,S}^{\leq 0}(\Sigma)$. For any such $\ol\CE$ the pairing is non-zero for a unique $m\geq 0$ and unique sequence of objects $\overline\CF_0,\ldots,\overline\CF_m$ as above and then its value is $\prod_{i=0}^m |\Aut(\overline\CF_i)|$.
Thus we have
\[
\sum_{\ol\CF\in\ParHiggs_{N,S}^{\leq 0}(\Sigma)/\sim} \frac{[\ol\CE]}{|\Aut(\CE,\theta)|} 
= \sum_{m=0}^\infty \sum_{\substack{\overline\CF_1,\ldots,\overline\CF_m\in\ParHiggs_{N,S}(\Sigma)/\sim \\ \text{semistable, $\mu\geq \mu_\alpha(\overline\CF_1)>\ldots>\mu_\alpha(\overline\CF_m)$}}} \frac{[\ol\CF_m]*\cdots*[\ol\CF_1]}{\prod_{i=1}^m |\Aut(\ol\CF_i)|}
\]
\[
* \sum_{\substack{\ol\CF_0\in\ParHiggs_{N,S}^{\leq 0}(\Sigma)/\sim\\
\mu_\alpha^\wmin(\ol\CF_0)>\mu}} \frac{[\ol\CF_0]}{|\Aut(\ol\CF_0)|} \qquad\text{(up to terms of rank $>r$)}.
\]
Note that for each product $[\overline\CF_m]*\cdots*[\overline\CF_0]$ we have $\Hom(\ol\CF_i, \ol\CF_{i'})=0$ for every $i<i'$ for slope reasons. Therefore we can apply Proposition \ref{prop:integration} and obtain
\[
\sum_{\ol\CF\in\ParHiggs_{N,S}^{\leq 0}(\Sigma)/\sim} \frac{w(\CE) q^{\chi(\CE,\CE)}}{|\Aut(\CE,\theta)|} = \prod_{\mu'\leq\mu} \left(1+\sum_{\substack{\ol\CF\in\ParHiggs_{N,S}(\Sigma)/\sim \\ \text{semistable, $\mu_\alpha(\CF)=\mu$}}} \frac{w(\CF) q^{\chi(\CF,\CF)}}{|\Aut(\ol\CF)|}\right)
\]
\[
\times \sum_{\substack{\ol\CF_0\in\ParHiggs_{N,S}^{\leq 0}(\Sigma)/\sim\\
\mu_\alpha^\wmin(\ol\CF_0)>\mu}} \frac{w(\CF_0) q^{\chi(\CF_0,\CF_0)}}{|\Aut(\ol\CF_0)|} \qquad\text{(up to terms of rank $>r$)}.
\]
The slope $\mu$ part of the expansion of the left hand side in the infinite product is given by $\pExp\left[\frac{q}{q-1}\BH_{g,k,\mu}[X_\bullet;T,q,t,\sigma]\right]$ by \eqref{eq:omega for higgs}. On the right hand side we see $\Omega_{\ParHiggs_{N,S}(\Sigma), \alpha,\mu}[X_\bullet;T,t]$
\end{proof}

Knowning the slope, the rank and the $r_{i,j}$ numbers uniquely determines the degree of a bundle. So we will drop the $t$ variable from $\Omega_{\ParHiggs_{N,S}(\Sigma), \alpha,\mu}[X_\bullet;T,t]$, which amounts to setting $t=1$. Note that twisting by a line bundle of degree $1$ gives an auto-equivalence of $\ParHiggs_{N,S}(\Sigma)$ which increases the slopes of all objects by $1$.
Thus we have
\begin{equation}\label{eq:shift mu}
\Omega_{\ParHiggs_{N,S}(\Sigma), \alpha,\mu}[X_\bullet;T] = \Omega_{\ParHiggs_{N,S}(\Sigma), \alpha,\mu+1}[X_\bullet;T].
\end{equation}
So we can calculate any coefficient of $\Omega_{\ParHiggs_{N,S}(\Sigma), \alpha,\mu}[X_\bullet;T]$ by first applying Proposition \ref{prop:omega slope mu} for $\mu'=\mu-m$ for a big enough $m$ and then using \eqref{eq:shift mu}.
 
\begin{cor}
The coefficients of $\Omega_{\ParHiggs_{N,S}(\Sigma), \alpha,\mu}[X_\bullet;T]$ are given by functions in $\BQ(q)[\alpha_1,\ldots,\alpha_{2g}]$. Thus we can define the \emph{Donaldson-Thomas invariants} as coefficients of the generating series 
\[
\BH_{\ParHiggs_{N,S}(\Sigma), \alpha,\mu}[X_\bullet;T] = \pLog\left[\Omega_{\ParHiggs_{N,S}(\Sigma), \alpha,\mu}[X_\bullet;T]\right].
\]
We have
\[
\BH_{\ParHiggs_{N,S}(\Sigma), \alpha,\mu}[X_\bullet;T] = \lim_{m\to\infty} \frac{q}{q-1}\BH_{g,k,\alpha,\mu-m}[X_\bullet;T,q,1,\sigma].
\]
\end{cor}

Combining this with \eqref{eq:polynomiality} we obtain
\begin{cor}
Let $r\in\BZ_{>0}$, $d\in\BZ$, $r_{i,j}\in\BZ_{\geq 0}$ for $i=1,\ldots,k$, $j=1,\ldots,N$. Let $\alpha$ be any stability condition. Then we have
\[
\BH_{\ParHiggs_{N,S}(\Sigma), \alpha,\mu}[X_\bullet;T] = \frac{q}{q-1}\BH_{g,k,\mu}^\HLV[X_\bullet;T,q,1,\sigma],
\]
where $\BH_{g,k,\mu}^\HLV$ is the sum of the terms of $\BH_{g,k}^\HLV$ whose slope $\mu'$ satisfies $\mu\equiv\mu'\pmod\BZ$.
\end{cor}

Note that in the following situation semistability implies stability over $\ol\BF_q$, and the number of automorphisms of such a stable object is $q-1$.
\begin{cor}\label{cor:counting generic}
Let $r\in\BZ_{>0}$, $d\in\BZ$, $r_{i,j}\in\BZ_{\geq 0}$ for $i=1,\ldots,k$, $j=1,\ldots,N$ be such that $\sum_{j=1}^N r_{i,j}=r$.
Let $\alpha$ be any stability condition which is generic for the data $r,d,r_{\bullet,\bullet}$, by which we mean that $\alpha_{i,j}>\alpha_{i,j+1}$, $\alpha_{i,N}>\alpha_{i,1}-1$ and for any $r', d', r_{\bullet,\bullet}'$ as above with $0<r'<r$ and $r_{i,j}'\leq r_{i,j}$ for all $i,j$ we have
\[
\frac{d'+\sum_{i=1}^k\sum_{j=1}^N \alpha_{i,j} r_{i,j}'}{r'}\neq \mu=\frac{d+\sum_{i=1}^k\sum_{j=1}^N \alpha_{i,j} r_{i,j}}{r}.
\]
Let
\[
\wdim=(2g-2+k) r^2 -\sum_{i=1}^k\sum_{j=1}^N r_{i,j}^2-2.
\]
Then the number of $\alpha$-stable parabolic
Higgs bundles $\ol\CE$ such that
\[
\rank\CE=r,\quad \deg\CE=d,\quad r_{i,j}(\CE)=r_{i,j}
\]
equals
\[
q^\frac\wdim2 \left(\BH_{g,k}^\HLV[X_\bullet;T,q,1,\sigma_\bullet]\Big|_{T^r}, \prod_{i=1}^k\prod_{j=1}^N h_{r_{i,j}}[X_i]\right),
\]
which is a polynomial in $q^{\pm 1}$ and $\alpha_\bullet$.
\end{cor}

\subsection{Poincar\'e polynomials of stable Higgs moduli spaces}
The moduli space of stable parabolic Higgs bundles was constructed by Yokogawa (see \cite{yokogawa1993compactification}, \cite{yokogawa1996moduli}). In the usual definition of a parabolic bundle we have a bundle $\CE$ and a flag 
\[
0=\CE_{i,0}\subset \CE_{i,1}\subset\cdots\supset\CE_{i,m_i}=\CE(s_i)
\]
with \emph{weights}
\[
1>\alpha_{i,1}>\cdots>\alpha_{i,m_i}\geq 0
\]
for each marked point $s_i$. This corresponds to the special case of our parabolic bundles satisfying $r_{i,j}=0$ for $j>m_i$ where $N$ is chosen larger than $m_i$ for all $i$. The choice of weights is encoded in the choice of stability condition so that the parabolic degree equals our $\deg_\alpha$. One can easily check that the notion of stability and semistability coincides with the usual one. Our parabolic Higgs bundles correspond to the usual notion of strict parabolic Higgs bundles, where the Higgs field has a simple pole at each marked point whose matrix residue is strictly block-upper-triangular with respect to the parabolic filtration.

It is well-known that under the genericity assumptions as in Corollary \ref{cor:counting generic} the corresponding moduli space $\CM$ is a smooth semiprojective variety of dimension 
\[
\dim\CM=(2g-2+k) r^2 -\sum_{i=1}^k\sum_{j=1}^N r_{i,j}^2-2
\]
By Corollary 1.3.2 of \cite{hausel2015cohomology} we know that the cohomology is pure (for more details see Proposition 3.7 in \cite{gothen2017topological}). Without loss of generality we assume $\Sigma$ is defined over a number field $F$. Then $\CM$ is also defined over $F$ and we can choose a model of $\CM$ over the ring of integers of $F$. Let $\mathfrak{p}$ be a prime of $F$ outside of the locus where the model is singular. By Deligne's theory of weights explained in \cite{deligne1975poids} we have that the Frobenius acting on the l-adic cohomology $H^i_l(\Sigma\otimes_F\ol F, \BQ_l)$ has eigenvalues of absolute value $q^{\frac{i}2}$. Choose an identification of the residue field of $\mathfrak{p}$ with $\BF_q$. We have the Lefschetz fixed point theorem, which says
\[
|\CM(\BF_{q^k})| = \sum_{i=0}^{2\dim \CM} (-1)^i \sum_{m=1}^{b_i} \alpha_{i,m}^k
\]
where $b_i=\dim H^i(\CM,\BC)$ and $|\alpha_{i,m}|=q^{\frac{2\dim \CM-i}2}$. Comparing this formula with Corollary \ref{cor:counting generic} we obtain
\begin{thm}\label{thm:formula poincare}
Suppose the rank $r$, degree $d$, the parabolic jumps $r_{\bullet,\bullet}$ and the stability condition $\alpha_{\bullet,\bullet}$ satisfy the genericity assumptions of Corollary \ref{cor:counting generic}. Let $\CM$ be the moduli space of stable parabolic Higgs bundles with corresponding data. Define the Poincar\'e polynomial by
\[
P(\CM,q) = \sum_{i=0}^{2\dim\CM} (-1)^i q^\frac{i}2 \dim H^i(\CM,\BC).
\]
Then we have
\[
P(\CM,q) = q^{\frac{\dim\CM}2}\left( \BH_{g,k}^\HLV[X_\bullet;T,q^{-1},1,q^{-\frac12},\ldots,{-q^\frac12}]\Big|_{T^r}, \prod_{i=1}^k\prod_{j=1}^N h_{r_{i,j}}[X_i]\right).
\]
\end{thm}

\subsection{Poincar\'e polynomials of character varieties}
To relate the result for Higgs moduli spaces to character varieties we recall Simpson's non-abelian Hodge theorem for non-compact curves \cite{simpson1990harmonic}.
Let $\Sigma$ be a Riemann surface of genus $g$ and let $S=(s_1,\ldots,s_k)$ be a collection of $k$ marked points on $\Sigma$. Fix an integer $n>0$ and $k$ conjugacy classes $C_1,\ldots,C_k$ in $\GL_r$. Simpson identifies the moduli space of irreducible representations $\pi_1(\Sigma, S)\to \GL_r(\BC)$ with local monodromies around the marked points given by $C_1,\ldots,C_k$ with the moduli space of parabolic stable Higgs bundles of rank $r$ of the following kind (see table on page 746 in \cite{simpson1990harmonic}). For each $i$ list the eigenvalues of $C_i$ without repetitions:
\[
v_{i,1},\ldots,v_{i,m_i}.
\]
Assume we have
\[
\arg v_{i,1}\leq \ldots \leq \arg v_{i,m_i}.
\]
Denote the multiplicities by $r_{i,1},\ldots,r_{i,m_i}$. Denote
\[
\alpha_{i,j}=-\frac{\arg v_{i,j}}{2\pi},\quad c_{i,j}=\frac{\log|v_{i,j}|}{4\pi}.
\]
Then we have to take stable parabolic Higgs bundles $\ol\CE$ with jumps given by $r_{i,j}$. The stability condition is given by $\alpha_{i,j}$. The Higgs field has simple poles, $\theta:\CE\to\CE\otimes\Omega^1_\Sigma(S)$. The residue matrices of $\theta$ are required to preserve the parabolic filtration. The action of $\theta$ on the respective graded components $\CE_{i,j}/\CE_{i,j-1}$ is required to have all eigenvalues equal to $\sqrt{-1} c_{i,j}$ and the Jordan form the same as the Jordan form of the eigenvalue $v_{i,j}$ part of $C_{i}$. The degree is uniquely determined from the condition that $\deg_\alpha\CE=0$:
\[
0 = \deg_\alpha\CE=\deg\CE + \sum_{i,j} r_{i,j} \alpha_{i,j},
\]
so
\[
\deg\CE=d=-\sum_{i,j} r_{i,j} \alpha_{i,j}.
\]

Note that if the character variety is not empty then we have
\[
\prod_{i,j} v_{i,j}^{r_{i,j}}=1
\]
for determinant reasons. This is equivalent to two conditions:
\[
\sum_{i,j} r_{i,j} \alpha_{i,j}\in\BZ, \qquad \sum_{i,j} r_{i,j} c_{i,j}=0.
\]
The first condition is equivalent to having $d\in\BZ$. The second condition translates into 
\[
\sum_{i=1}^k \res_{z=s_i} \Tr \theta = 0.
\]

The theory we have developed only covers the case when $\res\theta$ acts as zero on each graded component $\CE_{i,j}/\CE_{i,j-1}$. This corresponds to the case when $|v_{i,j}|=1$ for all $i,j$ and $C_i$ is the conjugacy class of a diagonal matrix. 

The genericity condition for character varieties is formulated as follows (see \cite{hausel2011arithmetic}):
\begin{defn}\label{defn:generic char var}
The data $r_{\bullet,\bullet}$, $v_{\bullet,\bullet}$ is generic if we have
\[
\prod_{i=1}^k\prod_{j=1}^{m_i} v_{i,j}^{r_{i,j}} = 1
\]
and for any $1\leq r'\leq r$ and any collection of numbers $r_{i,j}'$ with $\sum_{j=1}^{m_i}=r'$ for all $i$ and $r_{i,j}'\leq r_{i,j}$ for all $i,j$ we have 
\[
\prod_{i=1}^k\prod_{j=1}^{m_i} v_{i,j}^{r_{i,j}'} \neq 1.
\]
\end{defn}

Comparing this definition with the one in Corollary \ref{cor:counting generic} we see that the data $r_{\bullet,\bullet}$, $v_{\bullet,\bullet}$ is generic if and only if the corresponding data $r_{\bullet,\bullet}$, $d$, $\alpha_{\bullet,\bullet}$ is. Note that the set of generic data with $|v_{i,j}|=1$ for all $i,j$ is Zariski dense in the set of generic data without this restriction. Thus we can extend our result from the situation of generic data with $|v_{i,j}|=1$ to the general case of generic data:

\begin{thm}\label{thm:formula poincare charvar}
For arbitrary genus $g$ and number of marked points $k$ and any $k$-tuple of generic diagonal conjugacy classes $C_1,\ldots,C_k$ of $\GL_n$ the Poincar\'e polynomial
\[
P(\CM,q) = \sum_{i=0}^{2\dim\CM} (-1)^i q^\frac{i}2 \dim H^i(\CM,\BC)
\]
of the corresponding character variety $\CM$ is given by
\[
P(\CM,q) = q^{\frac{\dim\CM}2}\left( \BH_{g,k}^\HLV[X_\bullet;T,q^{-1},1,q^{-\frac12},\ldots,{-q^\frac12}]\Big|_{T^n}, \prod_{i=1}^k\prod_{j=1}^{m_i} h_{r_{i,j}}[X_i]\right),
\]
where $r_{i,1}, r_{i,2},\ldots, r_{i,m_i}$ are the multiplicities of the eigenvalues of $C_i$ for $i=1,\ldots,k$.
\end{thm}

\section*{Acknowledgements}
My work on this project started when I was a postdoc at IST Austria in the group of Tamas Hausel. My stay was supported by Advanced Grant ``Arithmetic and Physics of Higgs moduli spaces'' No. 320593 of the European Research Council. I thank Tamas Hausel for useful discussions and his group and IST Austria for stimulating environment. In the fall of 2017 Emmanuel Letellier visited IST Austria. He gave very interesting talks about counting nilpotent endomorphisms and my work was partially inspired by the talks and discussions with him.

I would also like to thank Olivier Schiffmann for useful discussions and his lectures on counting bundles at the workshop on Higgs bundles organized in 2017 at SISSA, Trieste. I also thank the organizers of this workshop.

I thank Fernando Rodriguez-Villegas for interesting discussions, from which I took many useful ideas about counting.

The main results were obtained and the paper was written at the University of Vienna, where I am supported by the Austrian Science
Fund (FWF) through the START-Project Y963-N35 of Michael Eichmair. I gratefully acknowledge their support.

\bibliographystyle{amsalpha}
\bibliography{refs}

\end{document}